\documentclass[11]{amsart}

\usepackage{amsmath,amssymb,fullpage,times}

\usepackage[shortlabels]{enumitem}

\newtheorem{thm}{Theorem}[section]
\newtheorem{cor}[thm]{Corollary}
\newtheorem{lem}[thm]{Lemma}
\theoremstyle{definition}
\newtheorem*{rems}{Remarks}
\newtheorem{defn}[thm]{Definition}
\theoremstyle{definition}
\theoremstyle{remark}

\numberwithin{equation}{section}

\begin{document}

\title{Variance of real zeros of random orthogonal polynomials}

\author{Doron S. Lubinsky}
\email{lubinsky@math.gatech.edu}
\address{School of Mathematics, Georgia Institute of Technology, Atlanta, GA 30332, USA}

\author{Igor E. Pritsker}
\email{igor@math.okstate.edu}
\address{Department of Mathematics, Oklahoma State University, Stillwater, OK 74078, USA}

\begin{abstract}
We determine the asymptotics for the variance of the number of zeros of
random linear combinations of orthogonal polynomials of degree
$\leq n$ in subintervals $\left [ a,b\right ] $ of the support of the underlying
orthogonality measure $\mu $. We show that, as
$n\rightarrow \infty $, this variance is asymptotic to $cn$, for some explicit
constant $c>0$.
\end{abstract}


\maketitle

\section{Introduction and main results}
\label{sec1}

Let $\mu $ be a positive Borel measure compactly supported in the real
line, whose support contains infinitely many points. For
$n\geq 0,\ n\in {\mathbb{Z%
}}$, we consider the $n$th orthonormal polynomial
%
%e1.1 #&#
\begin{equation}\label{eq1.1}
p_{n}\left ( x\right ) =\gamma _{n}x^{n}+...
\end{equation}%
for $\mu $, with $\gamma _{n}>0$, so that
\begin{equation*}
\int p_{n}(x)p_{m}(x)\,d\mu (x)=\delta _{mn},\quad m,n\geq 0.
\end{equation*}%
Define the ensemble of random orthogonal polynomials of the form
%
%e1.2 #&#
\begin{equation}\label{eq1.2}
G_{n}(x)=\sum _{j=0}^{n}a_{j}p_{j}(x),\quad n\geq 0,
\end{equation}%
where $\{a_{j}\}_{j=0}^{\infty }$ are standard Gaussian
$\mathcal{N}\left ( 0,1\right ) $ i.i.d. random variables. For any interval
$\left [ a,b\right ] \subset {\mathbb{R}}$, let
$N_{n}(\left [ a,b\right ] )$ (resp.
$N_{n}\left ( \mathbb{R}\right ) $) denote the number of zeros of
$G_{n}$ lying in $\left [ a,b\right ] $ (resp. total number of real zeros).

Real zeros of high degree random polynomials have been studied since the
1930s. The early work concentrated on the expected number of real zeros
${%
\mathbb{E}}[N_{n}({\mathbb{R}})]$ for
$P_{n}(x)=\sum _{k=0}^{n}a_{k}x^{k}$, where $\{a_{k}\}_{k=0}^{n}$ are i.i.d.
random variables. Bloch and P\'{o}lya \cite{BlochPolya1932} gave the upper
bound ${\mathbb{E}}[N_{n}({\mathbb{R}}%
)]=O(\sqrt{n})$ for polynomials with coefficients in $\{-1,0,1\}$. Improvements
and generalizations were obtained by Littlewood and Offord
\cite{LittlewoodOfford1938,LittlewoodOfford1939}, Erd\H{o}s and Offord
\cite{ErdosOfford1956} and others. Kac \cite{Kac1943} introduced the ``Kac-Rice
formula'' to establish the important asymptotic result
\begin{equation*}
{\mathbb{E}}[N_{n}({\mathbb{R}})]=(2/\pi +o(1))\log {n}\quad
\text{as }%
n\rightarrow \infty ,
\end{equation*}%
for polynomials with independent real Gaussian coefficients.

More precise forms of this asymptotic were obtained by Kac
\cite{Kac1948}, Edelman and Kostlan \cite{EdelmanKostlan1995}, Wilkins
\cite{Wilkins1988} and others. For related further directions, see
\cite{BharuchaReidSambandham1986} and \cite{Farahmand1998}. Maslova
\cite{Maslova1974} proved that the variance of real zeros for Kac polynomials
$%
\sum _{k=0}^{n}a_{k}z^{k}$ satisfies
\begin{equation*}
\text{\textup{Var}}[N_{n}({\mathbb{R}})]=\frac{4}{\pi }\left ( 1-
\frac{2}{\pi }\right ) \log {n}+o(\log {n})
\end{equation*}%
for i.i.d. coefficients with mean 0, variance 1 and
${\mathbb{P}}(a_{k}=0)=0$%
. This result was recently generalized by Nguyen and Vu
\cite{NguyenVu2020}.

Das \cite{Das1971} considered random Legendre polynomials corresponding
to Lebesgue measure $d\mu (x)=dx$ on $[-1,1]$, and found that
${\mathbb{E}}%
[N_{n}(\left [ -1,1\right ] )]$ is asymptotically equal to
$n/\sqrt{3}$. Wilkins \cite{Wilkins1997} estimated the error term in this
asymptotic relation. For random Jacobi polynomials, Das and Bhatt
\cite{DasBhatt1982} established that
${\mathbb{E}}[N_{n}(\left [ -1,1\right ] )]$ is asymptotically equal to
$n/\sqrt{3}$ too. Farahmand \cite{Farahmand1996},
\cite{Farahmand1998}, \cite{Farahmand2001} considered the expected number
of the level crossings of random sums of Legendre polynomials with coefficients
having different distributions. These results were generalized to wide
classes of random orthogonal polynomials by Lubinsky, Pritsker and Xie
\cite{Lubinskyetal2016} and \cite{Lubinskyetal2018}. In particular, they showed
that the first term in the asymptotics for
${\mathbb{E}}[N_{n}({\mathbb{R}}%
)] $ remains the same as for the Legendre case.

The asymptotic variance and the Gaussianity for real zeros of random trigonometric
polynomials were established by Granville and Wigman
\cite{GranvilleWigman2011}, and subsequently by Aza\"{\i}s and Le\'{o}n
\cite{AzaisLeon2013} via different methods. Su and Shao \cite{SuShao2012} found
the asymptotic variance for the real zeros of random cosine polynomials,
while Aza\"{\i}s, Dalmao and Le\'{o}n \cite{Azaisetal2016} gave a different
proof. Xie \cite{Xie2016} showed that the variance of real zeros for a
general class of random orthogonal polynomials is $o(n^{2})$. A recent
paper of Do, H. Nguyen and O. Nguyen \cite{Doetal2020} studied dependence
of the variance on the distribution of the i.i.d. random coefficients in
the trigonometric case.

In this paper our main goal is determining the asymptotic for the variance
of the number of real zeros for the ensemble of random orthogonal polynomials
of the form {(\ref{eq1.2})}. To state our results, we require the following definition:%

%d1.1 #&#
\begin{defn}\label{defn1.1}
We say that a measure is regular in the sense of Stahl, Totik, and
Ullman, if the leading coefficients
$\left \{  \gamma _{j}\right \}  $ of the orthonormal polynomials in {(\ref{eq1.1})} satisfy
\begin{equation*}
\lim _{j\rightarrow \infty }\gamma _{j}^{1/j}=
\frac{1}{cap\left ( \text{supp}%
\left [ \mu \right ] \right ) },
\end{equation*}%
where $cap\left ( \text{supp}\left [ \mu \right ] \right ) $
denotes the logarithmic capacity of $\mathrm{supp}\left [ \mu \right ] $.
\end{defn}

While not a transparent condition, it is a weak one. For example, if the
support of $\mu $ consists of finitely many intervals, and
$\mu ^{\prime }$ is positive a.e. in each of those intervals, then
$\mu $ is regular. However, much less is needed \cite{StahlTotik1992}.
We let $\nu $ denote the equilibrium measure $\nu $ for $\mathrm{supp}\left [
\mu \right ] $ in the sense of potential theory, and let
$\omega \left ( x\right ) =\frac{d\nu }{dx}$. In any open subinterval of
$\mathrm{supp}\left [ \mu \right ] $, $\omega $ exists, and is positive and continuous
\cite{StahlTotik1992}. For example, when $\mathrm{supp}\left [ \mu \right ] =
\left [ -1,1\right ] $,
\begin{equation*}
\omega \left ( x\right ) =\frac{1}{\pi \sqrt{1-x^{2}}}.
\end{equation*}

Let%
%
%e1.3 #&#
\begin{equation}\label{eq1.3}
S\left ( u\right ) =\frac{\sin \pi u}{\pi u};
\end{equation}%
%
%e1.4 #&#
\begin{equation}\label{eq1.4}
F\left ( u\right ) =\det \left [
\begin{array}{c@{\quad }c@{\quad }c@{\quad }c}
1 & S\left ( u\right ) & 0 & S^{\prime }\left ( u\right )
\\
S\left ( u\right ) & 1 & -S^{\prime }\left ( u\right ) & 0
\\
0 & -S^{\prime }\left ( u\right ) & -S^{\prime\prime }\left ( 0
\right ) & -S^{\prime\prime }\left ( u\right )
\\
S^{\prime }\left ( u\right ) & 0 & -S^{\prime\prime }\left ( u
\right ) & -S^{\prime\prime }\left ( 0\right )%
\end{array}%
\right ] ;
\end{equation}%
%
%e1.5 #&#
\begin{equation}\label{eq1.5}
G\left ( u\right ) =\det \left [
\begin{array}{c@{\quad }c@{\quad }c}
1 & S\left ( u\right ) & -S^{\prime }\left ( u\right )
\\
S\left ( u\right ) & 1 & 0
\\
-S^{\prime }\left ( u\right ) & 0 & -S^{\prime\prime }\left ( 0
\right )%
\end{array}%
\right ] ;
\end{equation}%
%
%e1.6 #&#
\begin{equation}\label{eq1.6}
H\left ( u\right ) =\det \left [
\begin{array}{c@{\quad }c@{\quad }c}
1 & S\left ( u\right ) & 0
\\
S\left ( u\right ) & 1 & -S^{\prime }\left ( u\right )
\\
S^{\prime }\left ( u\right ) & 0 & -S^{\prime\prime }\left ( u
\right )%
\end{array}%
\right ] .
\end{equation}%
Sylvester's determinant identity and the fact that
$G\left ( -u\right ) =G\left ( u\right ) $ show that
\begin{equation*}
\left ( 1-S\left ( u\right ) ^{2}\right ) F\left ( u\right ) =G\left (
u\right ) ^{2}-H\left ( u\right ) ^{2}.
\end{equation*}%
Also let%
%
%e1.7 #&#
\begin{equation}\label{eq1.7}
\Xi \left ( u\right ) =\frac{1}{\pi ^{2}}\left \{
\frac{\sqrt{F\left ( u\right ) }%
}{1-S\left ( u\right ) ^{2}}+
\frac{1}{\left ( 1-S\left ( u\right ) ^{2}\right )
^{3/2}}H\left ( u\right ) \arcsin \left (
\frac{H\left ( u\right ) }{G\left (
u\right ) }\right ) \right \}  -\frac{1}{3}
\end{equation}%
and
%
%e1.8 #&#
\begin{equation}\label{eq1.8}
c=\int _{-\infty }^{\infty }\Xi \left ( u\right ) du+
\frac{1}{\sqrt{3}}.
\end{equation}%
%
%t1.2 #&#
\begin{thm}\label{thm1.2}
Let $\mu $ be a measure with compact support on the real
line, that is regular in the sense of Stahl, Totik, and Ullmann. Let
$\omega $
denote the Radon-Nikodym derivative of the equilibrium measure for
the support of $\mu $. Let $\left [ a^{\prime },b^{\prime }
\right ] $ be a subinterval in the support of $\mu $,
such that $\mu $ is absolutely continuous there, and
its Radon-Nikodym derivative $\mu ^{\prime }$ is positive and
continuous there. Assume moreover, that
%
%e1.9 #&#
\begin{equation}\label{eq1.9}
\sup _{n\geq 1}\left \Vert p_{n}\right \Vert _{L_{\infty }\left [ a^{
\prime },b^{\prime }\right ] }<\infty .
\end{equation}%
If $\left [ a,b\right ] \subset \left ( a^{\prime },b^{
\prime }\right )$ then
%
%e1.10 #&#
\begin{equation}\label{eq1.10}
\lim _{n\rightarrow \infty }\frac{1}{n}\text{Var}\left [ N_{n}\left (
\left [ a,b%
\right ] \right ) \right ] =c\left ( \int _{a}^{b}\omega \left ( y
\right ) dy\right ) .
\end{equation}
\end{thm}

Note that the limit does not depend on the particular measure $\mu $, but
involves the equilibrium density of the support of $\mu $. The bounds for
the orthonormal polynomials are known for example when
$\mu ^{\prime }$ satisfies a Dini-Lipschitz condition. Therefore an application
of {Theorem~\ref{thm1.2}} gives:%

%c1.3 #&#
\begin{cor}\label{cor1.3}
Let $\mu $ be a measure supported on
$\left [ -1,1 \right ] $
satisfying the Szeg\H{o} condition
\begin{equation*}
\int _{-1}^{1}\log \mu ^{\prime }\left ( x\right )
\frac{dx}{\pi \sqrt{1-x^{2}}}%
>-\infty .
\end{equation*}%
Let $\left [ a^{\prime },b^{\prime }\right ] $ be a subinterval
of $\left ( -1,1\right ) $, in which $\mu $ is absolutely
continuous, while $\mu ^{\prime }$ is positive and continuous
in $\left [ a^{\prime },b^{\prime }\right ] $. Assume moreover
that its local modulus of continuity,
\begin{equation*}
\Omega \left ( t\right ) =\sup \left \{  \left \vert \mu ^{\prime }
\left ( x\right ) -\mu ^{\prime }\left ( y\right ) \right \vert :x,y
\in \left [ a^{\prime },b^{\prime }\right ] \text{ and }\left \vert x-y
\right \vert \leq t\right \}  ,%
\text{ }t>0,
\end{equation*}%
satisfies the Dini-Lipshitz condition
\begin{equation*}
\int _{0}^{1}\frac{\Omega \left ( t\right ) }{t}dt<\infty .
\end{equation*}%
If $\left [ a,b\right ] \subset \left ( a^{\prime },b^{
\prime }\right ) $, then
%
%e1.11 #&#
\begin{equation}\label{eq1.11}
\lim _{n\rightarrow \infty }\frac{1}{n}\text{Var}\left [ N_{n}\left (
\left [ a,b%
\right ] \right ) \right ] =c\left ( \int _{a}^{b}
\frac{1}{\pi \sqrt{1-y^{2}}}%
dy\right ) .
\end{equation}%
\end{cor}

\begin{rems}
\begin{enumerate}[(c)]
\item[(a)] We believe that this result is new even for the Legendre weight
$\mu ^{\prime }=1$.%

\item[(b)] The hypotheses of {Theorem~\ref{thm1.2}} are also satisfied for exponential weights
investigated in \cite{LevinLubinsky2001} that do not satisfy the Szeg\H{o}
condition. For example, the conclusion of {Theorem~\ref{thm1.2}} holds for any
$\left [ a,b\right ] \subset \left ( -1,1\right ) $, when
\begin{equation*}
\mu ^{\prime }\left ( x\right ) =\exp \left ( -\exp _{k}\left ( 1-x^{2}
\right ) ^{-\alpha }\right ) \text{, }x\in \left ( -1,1\right ),
\end{equation*}%
where $\alpha >0$ and
$\exp _{k}=\exp \left ( \exp \left ( ...\exp (\right ) )\right ) $ denotes
the $k$th iterated exponential.%

\item[(c)] For a class of weights supported on several disjoint intervals, in
a classic paper, Widom \cite{Widom1969} established asymptotics of the
orthonormal polynomials under some smoothness conditions on the weight.
These imply the uniform boundedness of the orthonormal polynomials in subintervals
of the interior of the support, so that {Theorem~\ref{thm1.2}} applies to Widom's
weights.%

\item[(d)] As noted above, the analogous limit for trigonometric polynomials was
established by Granville and Wigman in \cite{GranvilleWigman2011}. We have
indications that our results are related to those of
\cite{GranvilleWigman2011} via the same limiting Paley-Wiener process.%

\item[(e)] Aza\"{\i}s, Dalmao and Le\'{o}n \cite[Theorem 1]{Azaisetal2016} found
the asymptotics for the variance of zeros of random cosine polynomials
$%
\sum _{k=0}^{n}a_{k}\cos ky$ on $\left [ 0,\pi \right ] $. These random
cosine polynomials are equivalent to the random Chebyshev polynomials
$%
\sum _{k=0}^{n}a_{k}T_{k}\left ( x\right ) $ on
$\left [ -1,1\right ] $ by the change of variable $y=\arccos x$. Our asymptotic
variance result of {Theorem~\ref{thm1.2}} for the random Chebyshev polynomials agrees
with that of \cite[Theorem 1]{Azaisetal2016} for random cosine polynomials.%
\end{enumerate}
\end{rems}

This paper is organized as follows: in Section~\ref{sec2}, we state the Kac-Rice
formula for the variance, and prove {Theorem~\ref{thm1.2}} and {Corollary~\ref{cor1.3}}, deferring
technical details to later sections. In Section~\ref{sec3}, we record some technical
estimates and gather results from elsewhere. In Section~\ref{sec4}, we estimate
the ``tail term'' with
$\left \vert x-y\right \vert \geq \frac{\Lambda }{n}$ in the integral defining
the main term in the variance. In Section~\ref{sec5}, we handle the ``central term''
where $x$ and $y$ are close, which gives the dominant contribution to the
integral. In Section~\ref{sec6}, the appendix, we prove the formula for the variance.

In the sequel, $C,C_{1},C_{2},...$ denote constants independent of
$n,x,y$. The same symbol may be different in different occurrences.

\medskip

\noindent \textbf{Acknowledgments}

\medskip

The authors would like to acknowledge the input of Igor Wigman of King's College London.
He provided essential insight into the literature and ideas
for this paper. The authors would also like to thank a referee for finding
an error in the statement of {Lemma~\ref{lem3.2}}.

%s2 #&#
\section{The proofs of {Theorem~\ref{thm1.2}} and {Corollary~\ref{cor1.3}}}
\label{sec2}

We begin with the Kac-Rice formulas for the expectation and variance. These
involve the reproducing kernel%
%
%e2.1 #&#
\begin{equation}\label{eq2.1}
K_{n}\left ( x,y\right ) =\sum _{j=0}^{n-1}p_{j}\left ( x\right ) p_{j}
\left ( y\right )
\end{equation}%
and for nonnegative integers $r,s$, its derivatives%
%
%e2.2 #&#
\begin{equation}\label{eq2.2}
K_{n}^{\left ( r,s\right ) }\left ( x,y\right ) =\sum _{j=0}^{n-1}p_{j}^{
\left ( r\right ) }\left ( x\right ) p_{j}^{\left ( s\right ) }\left (
y\right ) .
\end{equation}%
%
%l2.1 #&#
\begin{lem}\label{lem2.1}
Let $\left [ a,b\right ] \subset \mathbb{R}$, and let $G_{n}$
be defined by {(\ref{eq1.2})}. Then the expected number of real zeros for
$G_{n}$ is expressed by
%
%e2.3 #&#
\begin{equation}\label{eq2.3}
{\mathbb{E}}\left [ N_{n}\left ( \left [ a,b\right ] \right ) \right ]
=\frac{1}{%
\pi }\int _{a}^{b}\rho _{1}\left ( x\right ) \text{ }dx,
\end{equation}%
where%
%
%e2.4 #&#
\begin{equation}\label{eq2.4}
\rho _{1}\left ( x\right ) =\frac{1}{\pi }\sqrt{
\frac{K_{n+1}^{(1,1)}\left (
x,x\right ) }{K_{n+1}\left ( x,x\right ) }-\left (
\frac{K_{n+1}^{(0,1)}\left (
x,x\right ) }{K_{n+1}\left ( x,x\right ) }\right ) ^{2}}.
\end{equation}
\end{lem}
\begin{proof}
See \cite{Lubinskyetal2016}.
\end{proof}

We note that $\rho _{1}$ depends on $n$, but we omit this dependence to
simplify the notation. The same applies to $\rho _{2}$ below. The variance
of real zeros of $G_{n}$ is found from the following formula, which was
derived in \cite{Xie2016} by using the method of
\cite{GranvilleWigman2011}.%

%l2.2 #&#
\begin{lem}\label{lem2.2}
Let $\left [ a,b\right ] \subset \mathbb{R}$, and let $G_{n}$
be defined by {(\ref{eq1.2})}. %
%
%e2.5 #&#
\begin{equation}\label{eq2.5}
\mathrm{Var}\left [ N_{n}\left ( \left [ a,b\right ] \right ) \right ]
=\int _{a}^{b}\int _{a}^{b}\left \{  \rho _{2}\left ( x,y\right ) -
\rho _{1}\left ( x\right ) \rho _{1}\left ( y\right ) \right \}  dxdy+
\int _{a}^{b}\rho _{1}\left ( x\right ) dx,
\end{equation}%
where
%
%e2.6 #&#
\begin{equation}\label{eq2.6}
\rho _{2}(x,y)=\frac{1}{\pi ^{2}\sqrt{\Delta }}\left ( \sqrt{\Omega _{11}
\Omega _{22}-\Omega _{12}^{2}}+\Omega _{12}\arcsin \left (
\frac{\Omega _{12}}{\sqrt{\Omega _{11}\Omega _{22}}}\right ) \right ) .
\end{equation}%
Here%
%
%e2.7 #&#
\begin{equation}\label{eq2.7}
\Delta (x,y):=K_{n+1}(x,x)K_{n+1}(y,y)-K_{n+1}^{2}(x,y)
\end{equation}%
and $\Omega $ is the covariance matrix of the random vector
$(P_{n}^{\prime }(x),P_{n}^{\prime }(y))$ conditional upon
$P_{n}(x)=P_{n}(y)=0$:
\begin{equation*}
\Omega =%
\begin{bmatrix}
\Omega _{11} & \Omega _{12}
\\
\Omega _{12} & \Omega _{22}%
\end{bmatrix}%
,
\end{equation*}%
with
%
%e2.8 #&#
\begin{align}
\label{eq2.8}
\Omega _{11}(x,y) &:= K_{n+1}^{(1,1)}(x,x)-
\nonumber
\\
 &\frac{1}{\Delta }\left ( K_{n+1}(y,y)(K_{n+1}^{(0,1)}(x,x))^{2}-2K_{n+1}(x,y)K_{n+1}^{(0,1)}(x,x)K_{n+1}^{(0,1)}(y,x)+K_{n+1}(x,x)(K_{n+1}^{(0,1)}(y,x))^{2}
\right ) ,
\end{align}
%
%e2.9 #&#
\begin{align}
\label{eq2.9}
\Omega _{22}(x,y) &:= K_{n+1}^{(1,1)}(y,y)-
\\
 &\frac{1}{\Delta }\left ( K_{n+1}(y,y)(K_{n+1}^{(0,1)}(x,y))^{2}-2K_{n+1}(x,y)K_{n+1}^{(0,1)}(x,y)K_{n+1}^{(0,1)}(y,y)+K_{n+1}(x,x)(K_{n+1}^{(0,1)}(y,y))^{2}
\right ) ,
\nonumber
\end{align}%
%e2.10 #&#
\begin{align}
\label{eq2.10}
\Omega _{12}(x,y) &:= K_{n+1}^{(1,1)}(x,y)-
\\
 &\frac{1}{\Delta }%
[K_{n+1}(y,y)K_{n+1}^{(0,1)}(x,x)K_{n+1}^{(0,1)}(x,y)-K_{n+1}(x,y)K_{n+1}^{(0,1)}(x,y)K_{n+1}^{(0,1)}(y,x)
\nonumber
\\
& -K_{n+1}(x,y)K_{n+1}^{(0,1)}(x,x)K_{n+1}^{(0,1)}(y,y)+K_{n+1}(x,x)K_{n+1}^{(0,1)}(y,x)K_{n+1}^{(0,1)}(y,y)].
\nonumber
\end{align}
\end{lem}

\begin{proof}
See the Appendix. It is also shown there that the matrix $\Omega $ is nonnegative
definite, so that the square root defining $\rho _{2}$ is well defined.
\end{proof}

To prove {Theorem~\ref{thm1.2}}, we split the first integral in {(\ref{eq2.5})} into a central
term that provides the main contribution, and a tail term: for some large
enough $\Lambda $, write
\begin{eqnarray*}
&&\int _{a}^{b}\int _{a}^{b}\left \{  \rho _{2}\left ( x,y\right ) -
\rho _{1}\left ( x\right ) \rho _{1}\left ( y\right ) \right \}  dx
\text{ }dy
\\
&=&\left [ \iint _{\left \{  \left ( x,y\right ) :x,y\in \left [ a,b
\right ] ,\left \vert x-y\right \vert \geq \Lambda /n\right \}  }+
\iint _{\left \{  \left ( x,y\right ) :x,y\in \left [ a,b\right ] ,
\left \vert x-y\right \vert <\Lambda /n\right \}  }\right ] \left \{
\rho _{2}\left ( x,y\right ) -\rho _{1}\left ( x\right ) \rho _{1}
\left ( y\right ) \right \}  dx\text{ }dy
\\
&=&\text{Tail }+\text{Central.}
\end{eqnarray*}%
We handle the tail term by proving the following estimate and a simple
consequence:

%l2.3 #&#
\begin{lem}\label{lem2.3}
\begin{enumerate}[(a)]
\item[\emph{(a)}] There exist $C_{1},n_{0}$, and $\Lambda _{0}$ such that for $n\geq n_{0}$ and
$\left \vert x-y\right \vert \geq \frac{\Lambda _{0}}{n}$,%
%
%e2.11 #&#
\begin{equation}\label{eq2.11}
\left \vert \rho _{2}\left ( x,y\right ) -\rho _{1}\left ( x\right )
\rho _{1}\left ( y\right ) \right \vert \leq
\frac{C_{1}}{\left \vert x-y\right \vert ^{2}}.
\end{equation}%
\item[\emph{(b)}] There exist $C_{2},n_{0}$, and $\Lambda _{0}$ such that for $n\geq n_{0}$ and
$\Lambda \geq \Lambda _{0}$,%
%
%e2.12 #&#
\begin{equation}\label{eq2.12}
\iint _{\left \{  \left ( x,y\right ) :x,y\in \left [ a,b\right ] ,
\left \vert x-y\right \vert \geq \Lambda /n\right \}  }\left \vert \rho _{2}\left ( x,y\right ) -\rho _{1}\left ( x\right ) \rho _{1}
\left ( y\right ) \right \vert dx%
\text{ }dy\leq C_{2}\frac{n}{\Lambda }.
\end{equation}
\end{enumerate}
\end{lem}
\begin{proof}
See Section~\ref{sec4}.
\end{proof}

Recall that $\Xi $ is defined by {(\ref{eq1.7})}. For the central term we will prove:%

%l2.4 #&#
\begin{lem}\label{lem2.4}
\begin{enumerate}[(b)]
\item[\emph{(a)}] Uniformly for $u$ in compact subsets of
$\mathbb{C}\backslash \left \{  0\right \}$, and $x\in \left [ a,b
\right ] $
and $y=x+\frac{u}{n\omega \left ( x\right ) }$,%
%
%e2.13 #&#
\begin{equation}\label{eq2.13}
\left ( \frac{1}{n\omega \left ( x\right ) }\right ) ^{2}\left \{
\rho _{2}\left ( x,y\right ) -\rho _{1}\left ( x\right ) \rho _{1}
\left ( y\right ) \right \}  =\Xi \left ( u\right ) +o\left ( 1\right )
.
\end{equation}%
\item[\emph{(b)}] Let $\eta >0$. There exists $C$ such that for
$x\in \left [ a,b\right ] $, $y=x+
\frac{u}{n\omega \left ( x\right ) } $,
$u\in \left [ -\eta , \eta \right ] $ and $n\geq 1$,
\begin{equation*}
\left \vert \rho _{2}\left ( x,y\right ) -\rho _{1}\left ( x\right )
\rho _{1}\left ( y\right ) \right \vert \leq Cn^{2}.
\end{equation*}
\end{enumerate}
\end{lem}

\begin{proof}%
See Section~\ref{sec5}.
\end{proof}

The second integral in {(\ref{eq2.5})} is simpler:%

%l2.5 #&#
\begin{lem}\label{lem2.5}
%e2.14 #&#
\begin{equation}\label{eq2.14}
\frac{1}{n}\int _{a}^{b}\rho _{1}\left ( x\right ) dx=
\frac{1}{\sqrt{3}}%
\int _{a}^{b}\omega \left ( x\right ) dx+o\left ( 1\right ) .
\end{equation}
\end{lem}

\begin{proof}%
See Section~\ref{sec5}.
\end{proof}

\begin{proof}[Proof of {Theorem~\ref{thm1.2}}]%
We fix $\Lambda >\eta >0$ and split%
%
%e2.15 #&#
\begin{eqnarray}\label{eq2.15}
&&\int _{a}^{b}\int _{a}^{b}\left \{  \rho _{2}\left ( x,y\right ) -
\rho _{1}\left ( x\right ) \rho _{1}\left ( y\right ) \right \}  dy
\text{ }dx
\nonumber
\\
&=&\int _{a}^{b}\left [ \int _{I}+\int _{J}+\int _{K}\right ] \left
\{  \rho _{2}\left ( x,y\right ) -\rho _{1}\left ( x\right ) \rho _{1}
\left ( y\right ) \right \}  dy\text{ }dx,
\end{eqnarray}%
where for a given $x$,
\begin{eqnarray*}
I &=&\left \{  y\in \left [ a,b\right ] :\left \vert y-x\right \vert \geq \Lambda /\left ( n\omega \left ( x\right ) \right ) \right \}  ;
\\
J &=&\left \{  y\in \left [ a,b\right ] :\eta /\left ( n\omega \left ( x
\right ) \right ) \leq \left \vert y-x\right \vert <\Lambda /\left ( n
\omega \left ( x\right ) \right ) \right \}  ;
\\
K &=&\left \{  y\in \left [ a,b\right ] :\left \vert y-x\right \vert <
\eta /\left ( n\omega \left ( x\right ) \right ) \right \}  .
\end{eqnarray*}%
If $\omega _{0}$ is the maximum of $\omega \left ( x\right ) $ in
$\left [ a,b%
\right ] $, (recall that $\omega $ is positive and continuous in
$\left [ a,b%
\right ] $) then
%
%e2.16 #&#
\begin{eqnarray}\label{eq2.16}
&&\left \vert \int _{a}^{b}\int _{I}\left \{  \rho _{2}\left ( x,y
\right ) -\rho _{1}\left ( x\right ) \rho _{1}\left ( y\right )
\right \}  dy\text{ }dx\right \vert
\nonumber
\\
&\leq &\iint _{\left \{  \left ( x,y\right ) :x,y\in \left [ a,b
\right ] ,\left \vert x-y\right \vert \geq \Lambda /\left ( n\omega _{0}
\right ) \right \}  }\left \vert \rho _{2}\left ( x,y\right ) -\rho _{1}
\left ( x\right ) \rho _{1}\left ( y\right ) \right \vert dy\text{ }dx
\nonumber
\\
&\leq &C_{1}\frac{n\omega _{0}}{\Lambda },
\end{eqnarray}
by {Lemma~\ref{lem2.3}}(b), provided $\Lambda /\omega _{0}\geq \Lambda _{0}$. Next,
\begin{eqnarray*}
&&\frac{1}{n}\int _{a}^{b}\int _{J}\left \{  \rho _{2}\left ( x,y
\right ) -\rho _{1}\left ( x\right ) \rho _{1}\left ( y\right )
\right \}  dy\text{ }dx
\\
&=&\int _{a}^{b}\omega \left ( x\right ) \int _{
\substack{ \eta \leq \left \vert u\right \vert \leq \Lambda , \\ x+\frac{u}{n\omega \left ( x\right ) }\in
\left [ a,b\right ] }}\left \{  \rho _{2}\left ( x,x+
\frac{u}{n\omega \left (
x\right ) }\right ) -\rho _{1}\left ( x\right ) \rho _{1}\left ( x+
\frac{u}{%
n\omega \left ( x\right ) }\right ) \right \}
\frac{1}{\left ( n\omega \left (
x\right ) \right ) ^{2}}du\text{ }dx.
\end{eqnarray*}%
Note that if $\eta \leq \left \vert u\right \vert \leq \Lambda $ and
$x\in
\left [ a,b\right ] $ but
$x+\frac{u}{n\omega \left ( x\right ) }\notin \left [ a,b\right ] $, then
$x$ is at a distance of $O\left ( \frac{\Lambda }{n}%
\right ) $ to $a$ or $b$, and in view of {Lemma~\ref{lem2.4}}(b), the integral over
such $\left ( x,u\right ) $ is $O\left ( \frac{1}{n}\right ) $. Using {Lemma~\ref{lem2.4}}(a), we deduce that
%
%e2.17 #&#
\begin{eqnarray}\label{eq2.17}
&&\lim _{n\rightarrow \infty }\frac{1}{n}\int _{a}^{b}\int _{J}\left
\{  \rho _{2}\left ( x,y\right ) -\rho _{1}\left ( x\right ) \rho _{1}
\left ( y\right ) \right \}  dydx
\nonumber
\\
&=&\left ( \int _{a}^{b}\omega \left ( x\right ) dx\right ) \text{ }
\left ( \int _{\eta \leq \left \vert u\right \vert \leq \Lambda }\Xi
\left ( u\right ) du\right ) .
\end{eqnarray}%
Finally, from {Lemma~\ref{lem2.4}}(b), (but with a different fixed $\eta $ there),%
%
%e2.18 #&#
\begin{equation}\label{eq2.18}
\frac{1}{n}\left \vert \int _{a}^{b}\int _{K}\left \{  \rho _{2}\left (
x,y\right ) -\rho _{1}\left ( x\right ) \rho _{1}\left ( y\right )
\right \}  dydx\right \vert \leq C\eta ,
\end{equation}%
where $C$ is independent of $n,\eta $. Combining the three estimates
{(\ref{eq2.16})}--{(\ref{eq2.18})}
over $I,J,K$, with {(\ref{eq2.15})} and {Lemma~\ref{lem2.5}}, we obtain
\begin{eqnarray*}
&&\limsup _{n\rightarrow \infty }\left \vert \frac{1}{n}Var\left [ N_{n}
\left ( \left [ a,b\right ] \right ) \right ] -\left ( \int _{a}^{b}
\omega \left ( x\right ) dx\right ) \text{ }\left ( \int _{\eta \leq
\left \vert u\right \vert \leq \Lambda }\Xi \left ( u\right ) du+
\frac{1}{\sqrt{3}}\right ) \right \vert \\
&\leq &C\left ( \frac{1}{\Lambda }+\eta \right ) ,
\end{eqnarray*}%
where $C$ is independent of $n,\Lambda ,\eta $. Now if
$B>A\geq \Lambda _{0}$, then {Lemma~\ref{lem2.3}}(b) and {Lemma~\ref{lem2.4}}(a) show that
\begin{eqnarray*}
&&\left ( \int _{a}^{b}\omega \left ( x\right ) dx\right ) \text{ }
\left \vert \int _{A\leq u\leq B}\Xi \left ( u\right ) du\right
\vert \\
&=&\lim _{n\rightarrow \infty }\frac{1}{n}\left \vert \int _{a}^{b}
\int _{\left \{  y\in \left [ a,b\right ] :A\omega \left ( x\right ) /n
\leq y-x<B\omega \left ( x\right ) /n\right \}  }\left \{  \rho _{2}
\left ( x,y\right ) -\rho _{1}\left ( x\right ) \rho _{1}\left ( y
\right ) \right \}  dy%
\text{ }dx\right \vert \leq C_{1}/A.
\end{eqnarray*}%
It follows that
$\int _{\Lambda _{0}}^{\infty }\Xi \left ( u\right ) du$ converges. Similarly,
$\int _{-\infty }^{-\Lambda _{0}}\Xi \left ( u\right ) du$ converges. So
we may let $\Lambda \rightarrow \infty $ above and deduce that
\begin{eqnarray*}
&&\limsup _{n\rightarrow \infty }\left \vert \frac{1}{n}Var\left [ N_{n}
\left ( \left [ a,b\right ] \right ) \right ] -\left ( \int _{a}^{b}
\omega \left ( x\right ) dx\right ) \text{ }\left ( \int _{\left
\vert u\right \vert \geq \eta }\Xi \left ( u\right ) du+
\frac{1}{\sqrt{3}}\right ) \right \vert \\
&\leq &C\eta .
\end{eqnarray*}%
On the other hand, {Lemma~\ref{lem2.4}}(a) and {Lemma~\ref{lem2.4}}(b) show that if
$0<\delta <\eta $,
\begin{eqnarray*}
&&\left ( \int _{a}^{b}\omega \left ( x\right ) dx\right ) \text{ }
\left \vert \int _{\delta \leq u\leq \eta }\Xi \left ( u\right ) du
\right \vert \\
&=&\lim _{n\rightarrow \infty }\frac{1}{n}\left \vert \int _{a}^{b}
\int _{\left \{  y\in \left [ a,b\right ] :\delta \omega \left ( x
\right ) /n\leq y-x<\eta \omega \left ( x\right ) /n\right \}  }\left
\{  \rho _{2}\left ( x,y\right ) -\rho _{1}\left ( x\right ) \rho _{1}
\left ( y\right ) \right \}  dy\text{ }dx\right \vert \leq C_{2}\eta .
\end{eqnarray*}%
It follows that $\int _{0}^{\eta }\Xi \left ( u\right ) du$ converges.
Similarly, $\int _{-\eta }^{0}\Xi \left ( u\right ) du$ converges. So we
may let $\eta \rightarrow 0+$ above to deduce the result.\vspace*{2pt}
\end{proof}

\begin{proof}[Proof of {Corollary~\ref{cor1.3}}]
Under the hypotheses of this theorem, Badkov even established asymptotics
for the orthonormal polynomials \cite[p. 42, Corollary 2]{Badkov1979} that
trivially imply {(\ref{eq1.9})}. Also, as noted above, since $\mu ^{\prime }$ satisfies
Szeg\H{o}'s condition and so is positive a.e. in
$\left [ -1,1%
\right ] $, it is regular \cite[Corollary 4.1.3]{StahlTotik1992}. Then
the result follows from {Theorem~\ref{thm1.2}}.\vspace*{2pt}
\end{proof}

%s3 #&#
\section{Auxiliary results}\vspace*{2pt}
\label{sec3}

Throughout this section, we assume that $\mu $ is as in {Theorem~\ref{thm1.2}}. We begin
by recording some determinantal and other formulae: let
$\Delta ,\Omega _{11},\Omega _{12},\Omega _{22}$ be as in {(\ref{eq2.7})}--{(\ref{eq2.10})}.
Also let
\begin{equation}\label{eq3.1}
\Sigma =\left [
\begin{array}{c@{\quad }c@{\quad }c@{\quad }c}
K_{n+1}\left ( x,x\right ) & K_{n+1}\left ( x,y\right ) & K_{n+1}^{
\left ( 0,1\right ) }\left ( x,x\right ) & K_{n+1}^{\left ( 0,1
\right ) }\left ( x,y\right )
\\
K_{n+1}\left ( x,y\right ) & K_{n+1}\left ( y,y\right ) & K_{n+1}^{
\left ( 0,1\right ) }\left ( y,x\right ) & K_{n+1}^{\left ( 0,1
\right ) }\left ( y,y\right )
\\
K_{n+1}^{\left ( 0,1\right ) }\left ( x,x\right ) & K_{n+1}^{\left ( 0,1
\right ) }\left ( y,x\right ) & K_{n+1}^{\left ( 1,1\right ) }\left ( x,x
\right ) & K_{n+1}^{\left ( 1,1\right ) }\left ( x,y\right )
\\
K_{n+1}^{\left ( 0,1\right ) }\left ( x,y\right ) & K_{n+1}^{\left ( 0,1
\right ) }\left ( y,y\right ) & K_{n+1}^{\left ( 1,1\right ) }\left ( x,y
\right ) & K_{n+1}^{\left ( 1,1\right ) }\left ( y,y\right )%
\end{array}%
\right ] .
\end{equation}

\begin{lem}\label{lem3.1}
\begin{enumerate}[(e)]
\item[\emph{(a)}]%
%
%e3.2 #&#
\begin{equation}\label{eq3.2}
\Delta \left ( x,y\right ) =\det \left [
\begin{array}{c@{\quad }c}
K_{n+1}\left ( x,x\right ) & K_{n+1}\left ( x,y\right )
\\
K_{n+1}\left ( y,x\right ) & K_{n+1}\left ( y,y\right )%
\end{array}%
\right ] ;
\end{equation}%
\item[\emph{(b)}] %
%
%e3.3 #&#
\begin{equation}\label{eq3.3}
\Delta \Omega _{11}=\det \left [
\begin{array}{c@{\quad }c@{\quad }c}
K_{n+1}\left ( y,y\right ) & K_{n+1}\left ( y,x\right ) & K_{n+1}^{
\left ( 0,1\right ) }\left ( y,x\right )
\\
K_{n+1}\left ( x,y\right ) & K_{n+1}\left ( x,x\right ) & K_{n+1}^{
\left ( 0,1\right ) }\left ( x,x\right )
\\
K_{n+1}^{\left ( 1,0\right ) }\left ( x,y\right ) & K_{n+1}^{\left ( 0,1
\right ) }\left ( x,x\right ) & K_{n+1}^{\left ( 1,1\right ) }\left ( x,x
\right )%
\end{array}%
\right ] ;
\end{equation}%
\item[\emph{(c)}]
%
%e3.4 #&#
\begin{equation}\label{eq3.4}
\Delta \Omega _{22}=\det \left [
\begin{array}{c@{\quad }c@{\quad }c}
K_{n+1}\left ( x,x\right ) & K_{n+1}\left ( x,y\right ) & K_{n+1}^{
\left ( 0,1\right ) }\left ( x,y\right )
\\
K_{n+1}\left ( y,x\right ) & K_{n+1}\left ( y,y\right ) & K_{n+1}^{
\left ( 0,1\right ) }\left ( y,y\right )
\\
K_{n+1}^{\left ( 1,0\right ) }\left ( y,x\right ) & K_{n+1}^{\left ( 1,0
\right ) }\left ( y,y\right ) & K_{n+1}^{\left ( 1,1\right ) }\left ( y,y
\right )%
\end{array}%
\right ] ;
\end{equation}%
\item[\emph{(d)}]
%
%e3.5 #&#
\begin{equation}\label{eq3.5}
\Delta \Omega _{12}=\det \left [
\begin{array}{c@{\quad }c@{\quad }c}
K_{n+1}\left ( x,x\right ) & K_{n+1}\left ( x,y\right ) & K_{n+1}^{
\left ( 0,1\right ) }\left ( x,x\right )
\\
K_{n+1}\left ( y,x\right ) & K_{n+1}\left ( y,y\right ) & K_{n+1}^{
\left ( 0,1\right ) }\left ( y,x\right )
\\
K_{n+1}^{\left ( 1,0\right ) }\left ( y,x\right ) & K_{n+1}^{\left ( 0,1
\right ) }\left ( y,y\right ) & K_{n+1}^{\left ( 1,1\right ) }\left ( y,x
\right )%
\end{array}%
\right ] .
\end{equation}%
\item[\emph{(e)}] Let $\Sigma $ be given by {(\ref{eq3.1})}. Then%
%
%e3.6 #&#
\begin{equation}\label{eq3.6}
\left ( \Omega _{11}\Omega _{22}-\Omega _{12}^{2}\right ) \Delta =
\det \left ( \Sigma \right ) .
\end{equation}
\end{enumerate}
\end{lem}
\begin{proof}

\noindent (a)--(d): These follow by expanding the determinants for example along
the bottom row.

\noindent (e) This can be established using Sylvester's determinant identity
\cite[p.
24, Thm. 1.4.1]{BakerGravesMorris1996} on the matrix $\Sigma $ defined
by {(\ref{eq3.1})}:
\begin{equation*}
\det \left ( \Sigma \right ) \det \left ( \Sigma _{3,4;3,4}\right ) =
\det \left ( \Sigma _{3;3}\right ) \det \left ( \Sigma _{4;4}\right ) -
\det \left ( \Sigma _{3;4}\right ) \det \left ( \Sigma _{4;3}\right ),
\end{equation*}%
where $\Sigma _{3,4;3,4}$ denotes the $2\times 2$ matrix formed from
$\Sigma $ by removing the 3rd and 4th rows and columns of $\Sigma $, while
$\Sigma _{r;s}$ denotes the $3\times 3$ matrix formed from $\Sigma $ by
removing the $r$th row and $s$th column. This identity and (a--d) yield
\begin{equation*}
\det \left ( \Sigma \right ) \Delta =\left ( \Delta \Omega _{22}
\right ) \left ( \Delta \Omega _{11}\right ) -\left ( \Delta \Omega _{12}
\right ) ^{2}.
\end{equation*}%
Note that in identifying $\det \left ( \Sigma _{4;4}\right ) $ with
$\Delta \Omega _{11}$, we have to swap the 1st and 2nd rows and columns.
Moreover, we use that $\Sigma _{4;3}^{T}=\Sigma _{3;4}$.
\end{proof}

Next, we record some estimates on the reproducing kernels and their derivatives:%

%l3.2 #&#
\begin{lem}\label{lem3.2}
Let $\left [ a,b\right ] $ be a subinterval of $
\left ( a^{\prime },b^{\prime }\right ) $. Then for $r,s=0,1$
and
$r=2,s=0$; and for all $n\geq 1$ and $x,y\in \left [ a,b%
\right ] $,
%
%e3.7 #&#
\begin{equation}\label{eq3.7}
\left \vert K_{n}^{\left ( r,s\right ) }\left ( x,y\right ) \right
\vert \leq \frac{Cn^{r+s}}{\left \vert x-y\right \vert +\frac{1}{n}}.
\end{equation}
\end{lem}
\begin{proof}%
First we note that since $\mu $ has compact support
\cite[p. 41]{Freud1971},
\begin{equation*}
C_{2}=\sup _{n\geq 1}\frac{\gamma _{n-1}}{\gamma _{n}}<\infty .
\end{equation*}%
The Christoffel-Darboux formula asserts that
\begin{equation*}
K_{n}\left ( x,y\right ) =\frac{\gamma _{n-1}}{\gamma _{n}}
\frac{p_{n}\left (
x\right ) p_{n-1}\left ( y\right ) -p_{n-1}\left ( x\right ) p_{n}\left ( y\right )
}{x-y},
\end{equation*}%
so that using our bound
$\left \vert p_{n}\left ( x\right ) \right \vert \leq C_{1}$ for
$x,y\in \left [ a^{\prime },b^{\prime }\right ] $,
\begin{equation*}
\left \vert K_{n}\left ( x,y\right ) \right \vert \leq
\frac{2C_{2}C_{1}^{2}}{%
\left \vert x-y\right \vert }.
\end{equation*}%
Moreover, by Cauchy-Schwartz,
\begin{equation*}
\left \vert K_{n}\left ( x,y\right ) \right \vert \leq \left ( \sum _{j=0}^{n-1}p_{j}^{2}
\left ( x\right ) \right ) ^{1/2}\left ( \sum _{j=0}^{n-1}p_{j}^{2}
\left ( y\right ) \right ) ^{1/2}\leq C_{1}^{2}n.
\end{equation*}%
Combining the last two inequalities gives
\begin{equation*}
\left \vert K_{n}\left ( x,y\right ) \right \vert \leq C_{1}^{2}\min
\left \{  \frac{2C_{2}}{\left \vert x-y\right \vert },n\right \},
\end{equation*}%
so that (for example, using the inequality between arithmetic and harmonic
means) we have the result {(\ref{eq3.7})} for $r=s=0$. Next,
%
%e3.8 #&#
\begin{eqnarray}\label{eq3.8}
&&K_{n}^{\left ( 1,0\right ) }\left ( x,y\right )
\nonumber
\\
&=&\frac{\gamma _{n-1}}{\gamma _{n}}\left (
\frac{p_{n}^{\prime }\left (
x\right ) p_{n-1}\left ( y\right ) -p_{n-1}^{\prime }\left ( x\right )
p_{n}\left ( y\right ) }{x-y}+
\frac{p_{n-1}\left ( x\right ) p_{n}\left (
y\right ) -p_{n-1}\left ( y\right ) p_{n}\left ( x\right ) }{\left ( x-y\right )
^{2}}\right ) .
\end{eqnarray}
To estimate the derivatives, we use Bernstein's inequality for derivatives,
namely for polynomials of degree $\leq n$,
\begin{equation*}
\left \vert P^{\prime }\left ( x\right ) \right \vert \leq
\frac{n}{\sqrt{1-x^{2}%
}}\left \Vert P\right \Vert _{L_{\infty }\left [ -1,1\right ] },x\in
\left ( -1,1\right ) .
\end{equation*}%
This has the following consequence: for $j,n\geq 1$ and polynomials
$P$ of degree $\leq n$,
\begin{equation*}
\left \Vert P^{(j)}\right \Vert _{L_{\infty }\left [ a,b\right ] }
\leq C_{3}n^{j}\left \Vert P\right \Vert _{L_{\infty }\left [ a^{
\prime },b^{\prime }%
\right ] }.
\end{equation*}%
Here $C_{3}$ depends on $j,a,b,a^{\prime },b^{\prime }$ but not on
$P$ nor on the degree $n$ of $P$. It then follows that for
$j=0,1,2$,
\begin{equation*}
C_{4}=\sup _{n\geq 1}\left \Vert p_{n}^{\left ( j\right ) }\right
\Vert _{L_{\infty }\left [ a,b\right ] }/n^{j}<\infty .
\end{equation*}%
Also then, from {(\ref{eq3.8})}, for $x,y\in \left [ a,b\right ] $,
\begin{equation*}
\left \vert K_{n}^{\left ( 1,0\right ) }\left ( x,y\right ) \right
\vert \leq 2C_{2}\left \{
\frac{C_{1}C_{4}n}{\left \vert x-y\right \vert }+
\frac{C_{1}^{2}%
}{\left \vert x-y\right \vert ^{2}}\right \}  .
\end{equation*}%
Next, by Cauchy-Schwartz,
\begin{equation*}
\left \vert K_{n}^{\left ( 1,0\right ) }\left ( x,y\right ) \right
\vert \leq \left ( \sum _{j=0}^{n-1}p_{j}^{\prime }\left ( x\right ) ^{2}
\right ) ^{1/2}\left ( \sum _{j=0}^{n-1}p_{j}^{2}\left ( y\right )
\right ) ^{1/2}\leq C_{4}C_{1}n^{2}.
\end{equation*}%
Thus
\begin{equation*}
\left \vert K_{n}^{\left ( 1,0\right ) }\left ( x,y\right ) \right
\vert \leq C_{5}\min \left \{  \frac{n}{\left \vert x-y\right \vert }+
\frac{1}{\left \vert x-y\right \vert ^{2}},n^{2}\right \}  .
\end{equation*}%
This yields {(\ref{eq3.7})} for $r=1,s=0$. Of course $r=0,s=1$ follows by symmetry.
Finally,
\begin{eqnarray*}
K_{n}^{\left ( 1,1\right ) } (x,y) &=&\frac{\gamma _{n-1}}{\gamma _{n}} \left(
\frac{p_{n}^{\prime }\left ( x\right ) p_{n-1}^{\prime }\left ( y\right )
-p_{n-1}^{\prime }\left ( x\right ) p_{n}^{\prime }\left ( y\right ) }{x-y}+%
\frac{p_{n}^{\prime }\left ( x\right ) p_{n-1}\left ( y\right ) -p_{n-1}^{\prime }\left ( x\right ) p_{n}\left ( y\right ) }{\left ( x-y\right ) ^{2}} \right.
\\
&&\left. +
\frac{p_{n-1}\left ( x\right ) p_{n}^{\prime }\left ( y\right )
-p_{n-1}^{\prime }\left ( y\right ) p_{n}\left ( x\right ) }{\left ( x-y\right )
^{2}}+2
\frac{p_{n-1}\left ( x\right ) p_{n}\left ( y\right ) -p_{n-1}\left (
y\right ) p_{n}\left ( x\right ) }{\left ( x-y\right ) ^{3}} \right).
\end{eqnarray*}%
Thus using our bounds on
$\left \{  p_{n}^{\left ( j\right ) }\right \}  $, $%
j=0,1,2$, gives for $x,y\in \left [ a,b\right ] $,
\begin{equation*}
\left \vert K_{n}^{\left ( 1,1\right ) }\left ( x,y\right ) \right
\vert \leq C_{6}\left \{  \frac{n^{2}}{\left \vert x-y\right \vert }+
\frac{n}{\left \vert x-y\right \vert ^{2}}+
\frac{1}{\left \vert x-y\right \vert ^{3}}\right \},
\end{equation*}%
and again Cauchy-Schwartz gives
\begin{equation*}
\left \vert K_{n}^{\left ( 1,1\right ) }\left ( x,y\right ) \right
\vert \leq \left ( \sum _{j=0}^{n-1}p_{j}^{\prime }\left ( x\right ) ^{2}
\right ) ^{1/2}\left ( \sum _{j=0}^{n-1}p_{j}^{\prime }\left ( x
\right ) ^{2}\right ) ^{1/2}\leq C_{7}n^{3}.
\end{equation*}%
This and the previous inequality give {(\ref{eq3.7})} for $r=s=1$. The case
$r=2,s=0$ is similar.
\end{proof}

Next, we record some universality limits. Recall that $S$ is defined by
{(\ref{eq1.3})}:%

%l3.3 #&#
\begin{lem}\label{lem3.3}
Let $\left [ a^{\prime },b^{\prime }\right ] $
be a subinterval in the support of $\mu $ such that
$\mu $ is absolutely continuous there, and $\mu ^{\prime }$ is
positive and continuous there. Let
$\left [ a,b\right ] \subset \left ( a^{\prime },b^{\prime }\right ) $.
Let $r,s$ be non-negative integers. Then%
\begin{enumerate}[\emph{(a)}]
\item[\emph{(a)}] Uniformly for $x\in \left [ a,b\right ] $ and $u,v$
in compact subsets of $\mathbb{C}$,
%
%e3.9 #&#
\begin{equation}\label{eq3.9}
\lim _{n\rightarrow \infty }
\frac{K_{n}^{\left ( r,s\right ) }\left ( x+\frac{u}{%
n\omega \left ( x\right ) },x+\frac{v}{n\omega \left ( x\right ) }\right ) }{%
K_{n}\left ( x,x\right ) }\left (
\frac{1}{n\omega \left ( x\right ) }\right ) ^{r+s}=\left ( -1\right )
^{s}S^{\left ( r+s\right ) }\left ( u-v\right ) .
\end{equation}%
\item[\emph{(b)}] Let%
%
%e3.10 #&#
\begin{equation}\label{eq3.10}
\tau _{r,s}=\left \{
\begin{array}{r@{\quad }r}
0, & r+s\text{ odd}
\\
\frac{\left ( -1\right ) ^{\left ( r-s\right ) /2}}{r+s+1}, & r+s
\text{ even}%
\end{array}%
\right . .
\end{equation}%
Then uniformly for $x\in \left [ a,b\right ] $,%
%
%e3.11 #&#
\begin{equation}\label{eq3.11}
\lim _{n\rightarrow \infty }\frac{1}{n^{r+s+1}}K_{n}^{\left ( r,s
\right ) }\left ( x,x\right ) \mu ^{\prime }\left ( x\right ) =\pi ^{r+s}
\omega \left ( x\right ) ^{r+s+1}\tau _{r,s}
\end{equation}%
and
%
%e3.12 #&#
\begin{equation}\label{eq3.12}
\lim _{n\rightarrow \infty }\frac{1}{n^{r+s}}
\frac{K_{n}^{\left ( r,s\right )
}\left ( x,x\right ) }{K_{n}\left ( x,x\right ) }=\left ( \pi \omega
\left ( x\right ) \right ) ^{r+s}\tau _{r,s}.
\end{equation}%
\item[\textit{(c)}] In particular, uniformly for
$x\in \left [ a,b \right ] $,
%
%e3.13 #&#
\begin{equation}\label{eq3.13}
\lim _{n\rightarrow \infty }\frac{1}{n^{2}}K_{n}^{\left ( 1,0\right ) }
\left ( x,x\right ) =0
\end{equation}%
and for $r=0,1$,
%
%e3.14 #&#
\begin{equation}\label{eq3.14}
K_{n}^{\left ( r,r\right ) }\left ( x,x\right ) \geq Cn^{2r+1}.
\end{equation}%
\item[\textit{(d)}]
%
%e3.15 #&#
\begin{equation}\label{eq3.15}
S^{\prime\prime }\left ( 0\right ) =-\frac{\pi ^{2}}{3}.
\end{equation}
\end{enumerate}
\end{lem}
\begin{proof}%
(a) We start with a result of Totik \cite[Theorem 2.2]{Totik2009}: uniformly
for $x\in \left [ a,b\right ] $, and $u,v$ in compact subsets of
$\mathbb{R}$,
%
%e3.16 #&#
\begin{equation}\label{eq3.16}
\lim _{n\rightarrow \infty }\frac{1}{n}K_{n}\left ( x+\frac{u}{n},x+
\frac{v}{n}%
\right ) \mu ^{\prime }\left ( x\right ) /\omega \left ( x\right ) =S
\left ( \left ( u-v\right ) \omega \left ( x\right ) \right ) .
\end{equation}%
In particular, it then follows that uniformly for
$x\in \left [ a,b\right ] $, and $u$ in compact subsets of
$\mathbb{R}$,
\begin{equation*}
\lim _{n\rightarrow \infty }
\frac{K_{n}\left ( x+\frac{u}{n},x+\frac{u}{n}%
\right ) }{K_{n}\left ( x,x\right ) }=1.
\end{equation*}%
Theorem 1.1 in \cite[p. 375]{Lubinsky2008} then asserts that uniformly
for $%
x\in \left [ a,b\right ] $, and $u,v$ in compact subsets of
$\mathbb{C}$,
\begin{equation*}
\lim _{n\rightarrow \infty }
\frac{K_{n}\left ( x+\frac{u}{K_{n}\left (
x,x\right ) \mu ^{\prime }\left ( x\right ) },x+\frac{v}{K_{n}\left ( x,x\right )
\mu ^{\prime }\left ( x\right ) }\right ) }{K_{n}\left ( x,x\right ) }=S
\left ( u-v\right ) .
\end{equation*}%
Here the uniformity and Totik's {(\ref{eq3.16})} allows us to replace
$K_{n}\left ( x,x\right ) \mu ^{\prime }\left ( x\right ) $ by
$n\omega \left ( x\right ) $: uniformly for
$x\in \left [ a,b\right ] $, and $u,v$ in compact subsets of
$%
\mathbb{C}$,
%
%e3.17 #&#
\begin{equation}\label{eq3.17}
\lim _{n\rightarrow \infty }
\frac{K_{n}\left ( x+\frac{u}{n\omega \left (
x\right ) },x+\frac{v}{n\omega \left ( x\right ) }\right ) }{K_{n}\left (
x,x\right ) }=S\left ( u-v\right ) .
\end{equation}%
This is the case $r=s=0$ of {(\ref{eq3.9})}. Because the limit holds uniformly for
$%
u,v $ in compact subsets of $\mathbb{C}$, we may differentiate this asymptotic
with respect to $u,v$ to get the general case of {(\ref{eq3.9})}.%

\noindent (b) For the special case where the support of $\mu $ is
$\left [ -1,1\right ]$,
this is {Corollary~\ref{cor1.3}} in \cite[p. 917]{Lubinsky2009} (see also \cite{Totik2000A}). There it was shown
that \cite[p. 937]{Lubinsky2009}%
%
%e3.18 #&#
\begin{equation}\label{eq3.18}
S\left ( u-v\right ) =\sum _{j,k=0}^{\infty }\frac{u^{j}}{j!}
\frac{v^{k}}{k!}%
\pi ^{j+k}\tau _{j,k},
\end{equation}%
so we can reformulate {(\ref{eq3.9})} for $r=s=0$ as
\begin{equation*}
\lim _{n\rightarrow \infty }\sum _{j,k=0}^{\infty }
\frac{\left ( \frac{u}{%
n\omega \left ( x\right ) }\right ) ^{j}}{j!}
\frac{\left ( \frac{v}{n\omega \left ( x\right ) }\right ) ^{k}}{k!}
\frac{K_{n}^{\left ( j,k\right ) }\left (
x,x\right ) }{K_{n}\left ( x,x\right ) }=\sum _{j,k=0}^{\infty }
\frac{u^{j}}{j!}%
\frac{v^{k}}{k!}\pi ^{j+k}\tau _{j,k}.
\end{equation*}%
Comparing coefficients of like powers of $u,v$ gives {(\ref{eq3.12})}.  That this
holds uniformly in $x$ for a given $r,s$ follows easily from the uniformity
of the original limit in $x$ (cf. \cite[p. 938]{Lubinsky2009}). Finally
Totik's limit {(\ref{eq3.16})} gives
\begin{equation*}
\lim _{n\rightarrow \infty }\frac{1}{n}K_{n}\left ( x,x\right ) \mu ^{
\prime }\left ( x\right ) /\omega \left ( x\right ) =1,
\end{equation*}%
uniformly for $x\in \left [ a,b\right ] $, so we also obtain the first
asymptotic {(\ref{eq3.11})}.

\noindent (c) This follows directly from (b).%

\noindent (d) From {(\ref{eq3.18})},
%
%e3.19 #&#
\begin{equation}\label{eq3.19}
S\left ( u\right ) =\sum _{j=0}^{\infty }\frac{u^{j}}{j!}\pi ^{j}
\tau _{j,0}.
\end{equation}%
So
$S^{\prime\prime }\left ( 0\right ) =\pi ^{2}\tau _{2,0}=-
\frac{\pi ^{2}}{3%
}$.
\end{proof}

%s4 #&#
\section{The tail term - {Lemma~\ref{lem2.3}}}
\label{sec4}

Recall that $\rho _{1},\rho _{2}$ are defined by {(\ref{eq2.4})} and {(\ref{eq2.6})}. First
write
%
%e4.1 #&#
\begin{equation}\label{eq4.1}
\rho _{1}\left ( x\right ) =\frac{1}{\pi K_{n+1}\left ( x,x\right ) }
\sqrt{\Psi \left ( x\right ) }
\end{equation}%
where%
%
%e4.2 #&#
\begin{equation}\label{eq4.2}
\Psi \left ( x\right ) =K_{n+1}^{\left ( 1,1\right ) }\left ( x,x
\right ) K_{n+1}\left ( x,x\right ) -K_{n+1}^{\left ( 0,1\right ) }
\left ( x,x\right ) ^{2}.
\end{equation}%
Next, write
%
%e4.3 #&#
\begin{equation}\label{eq4.3}
\rho _{2}\left ( x,y\right ) -\rho _{1}\left ( x\right ) \rho _{1}
\left ( y\right ) =T_{1}+T_{2}+T_{3},
\end{equation}%
where%
%
%e4.4 #&#
\begin{eqnarray}\label{eq4.4}
T_{1} &=&\frac{1}{\pi ^{2}\Delta }\left ( \sqrt{\left ( \Omega _{11}
\Omega _{22}-\Omega _{12}^{2}\right ) \Delta }-\sqrt{\Psi \left ( x
\right ) \Psi \left ( y\right ) }\right ) ;
\nonumber
\\
T_{2} &=&\frac{1}{\pi ^{2}\sqrt{\Delta }}\left \vert \Omega _{12}
\right \vert \arcsin \left (
\frac{\left \vert \Omega _{12}\right \vert }{\sqrt{\Omega _{11}\Omega _{22}}}
\right ) ;
\nonumber
\\
T_{3} &=&\frac{1}{\pi ^{2}}\left ( \frac{1}{\Delta }-
\frac{1}{K_{n+1}\left (
x,x\right ) K_{n+1}\left ( y,y\right ) }\right ) \sqrt{\Psi \left ( x
\right ) \Psi \left ( y\right ) }.
\end{eqnarray}%
We estimate each $T$ term separately. It is the following lemma that contains
the main idea, namely cancellation using Laplace's determinant formula:%

%l4.1 #&#
\begin{lem}\label{lem4.1}
There exist $n_{0}$ and $\Lambda _{0}>0$ such that
for $n\geq n_{0}$ and all $x,y\in \left [ a,b\right ]
$, with $\left \vert x-y\right \vert \geq \Lambda _{0}/n$,%
%
%e4.5 #&#
\begin{equation}\label{eq4.5}
\left \vert T_{1}\right \vert \leq
\frac{C}{\left ( \left \vert x-y\right \vert +%
\frac{1}{n}\right ) ^{2}}.
\end{equation}
\end{lem}
\begin{proof}
Write
\begin{equation*}
T_{1}=
\frac{\left ( \Omega _{11}\Omega _{22}-\Omega _{12}^{2}\right ) \Delta -\Psi \left ( x\right ) \Psi \left ( y\right ) }{\pi ^{2}\Delta \left [ \sqrt{%
\left ( \Omega _{11}\Omega _{22}-\Omega _{12}^{2}\right ) \Delta }+\sqrt{\Psi \left ( x\right ) \Psi \left ( y\right ) }\right ] }=
\frac{\text{Num}}{\text{Denom}}.
\end{equation*}%
The numerator is (recall {(\ref{eq3.6})})
\begin{eqnarray*}
\text{Num} &=&\left ( \Omega _{11}\Omega _{22}-\Omega _{12}^{2}\right )
\Delta -\Psi \left ( x\right ) \Psi \left ( y\right )
\\
&=&\det \left ( \Sigma \right ) -\Psi \left ( x\right ) \Psi \left ( y
\right )
\\
&=&\det \left [
\begin{array}{c@{\quad }c@{\quad }c@{\quad }c}
K_{n+1}\left ( x,x\right ) & K_{n+1}\left ( x,y\right ) & K_{n+1}^{
\left ( 0,1\right ) }\left ( x,x\right ) & K_{n+1}^{\left ( 0,1
\right ) }\left ( x,y\right )
\\
K_{n+1}\left ( x,y\right ) & K_{n+1}\left ( y,y\right ) & K_{n+1}^{
\left ( 0,1\right ) }\left ( y,x\right ) & K_{n+1}^{\left ( 0,1
\right ) }\left ( y,y\right )
\\
K_{n+1}^{\left ( 0,1\right ) }\left ( x,x\right ) & K_{n+1}^{\left ( 0,1
\right ) }\left ( y,x\right ) & K_{n+1}^{\left ( 1,1\right ) }\left ( x,x
\right ) & K_{n+1}^{\left ( 1,1\right ) }\left ( x,y\right )
\\
K_{n+1}^{\left ( 0,1\right ) }\left ( x,y\right ) & K_{n+1}^{\left ( 0,1
\right ) }\left ( y,y\right ) & K_{n+1}^{\left ( 1,1\right ) }\left ( x,y
\right ) & K_{n+1}^{\left ( 1,1\right ) }\left ( y,y\right )%
\end{array}%
\right ]
\\
&&-\det \left [
\begin{array}{c@{\quad }c}
K_{n+1}\left ( x,x\right ) & K_{n+1}^{\left ( 0,1\right ) }\left ( x,x
\right )
\\
K_{n+1}^{\left ( 0,1\right ) }\left ( x,x\right ) & K_{n+1}^{\left ( 1,1
\right ) }\left ( x,x\right )%
\end{array}%
\right ] \det \left [
\begin{array}{c@{\quad }c}
K_{n+1}\left ( y,y\right ) & K_{n+1}^{\left ( 0,1\right ) }\left ( y,y
\right )
\\
K_{n+1}^{\left ( 0,1\right ) }\left ( y,y\right ) & K_{n+1}^{\left ( 1,1
\right ) }\left ( y,y\right )%
\end{array}%
\right ] .
\end{eqnarray*}%
Let $\Sigma $ be the $4\times 4$ matrix above. Then we can write this as
\begin{equation*}
\text{Num}=\det \left [ \Sigma \right ] -\det \left [ \Sigma \left (
\begin{array}{c@{\quad }c}
1 & 3
\\
1 & 3%
\end{array}%
\right ) \right ] \det \left [ \Sigma \left (
\begin{array}{c@{\quad }c}
2 & 4
\\
2 & 4%
\end{array}%
\right ) \right ]
\end{equation*}%
where
$\Sigma \left (
\begin{array}{c@{\quad }c}
r & s
\\
j & k%
\end{array}%
\right ) $ denotes the matrix formed from $\Sigma $ by taking the elements
that lie in rows $r,s$ and columns $j,k$. Now let us use Laplace's determinant
expansion \cite[p. 37]{LancasterTismenetsky1985}: we have chosen rows
$1,3$. Laplace's expansion gives
\begin{equation*}
\det \left ( \Sigma \right ) =\sum _{1\leq j<k\leq 4}\left ( -1
\right ) ^{1+3+j+k}\det \left [ \Sigma \left (
\begin{array}{c@{\quad }c}
1 & 3
\\
j & k%
\end{array}%
\right ) \right ] \det \left [ \Sigma ^{c}\left (
\begin{array}{c@{\quad }c}
1 & 3
\\
j & k%
\end{array}%
\right ) \right ] ,
\end{equation*}%
where $\Sigma ^{c}$ is formed from the complimentary rows and columns.
The choices for $\left ( j,k\right ) $ are
$\left ( 1,2\right ), \left ( 1,3\right ), \left ( 1,4 \right ), \left ( 2,3\right ),$ 
$\left ( 2,4\right ), \left ( 3,4\right ). $ This gives $\det \left ( \Sigma \right ) $ as a sum of 6 terms,
one of which is
$\det \left [ \Sigma \left (
\begin{array}{c@{\quad }c}
1 & 3
\\
1 & 3%
\end{array}%
\right ) \right ] \det \left [ \Sigma \left (
\begin{array}{c@{\quad }c}
2 & 4
\\
2 & 4%
\end{array}%
\right ) \right ] $. So
\begin{eqnarray*}
\text{Num}&=&-\det \left [ \Sigma \left (
\begin{array}{c@{\quad }c}
1 & 3
\\
1 & 2%
\end{array}%
\right ) \right ] \det \left [ \Sigma \left (
\begin{array}{c@{\quad }c}
2 & 4
\\
3 & 4%
\end{array}%
\right ) \right ]
\\
&&-\det \left [ \Sigma \left (
\begin{array}{c@{\quad }c}
1 & 3
\\
1 & 4%
\end{array}%
\right ) \right ] \det \left [ \Sigma \left (
\begin{array}{c@{\quad }c}
2 & 4
\\
2 & 3%
\end{array}%
\right ) \right ] -\det \left [ \Sigma \left (
\begin{array}{c@{\quad }c}
1 & 3
\\
2 & 3%
\end{array}%
\right ) \right ] \det \left [ \Sigma \left (
\begin{array}{c@{\quad }c}
2 & 4
\\
1 & 4%
\end{array}%
\right ) \right ]
\\
&&+\det \left [ \Sigma \left (
\begin{array}{c@{\quad }c}
1 & 3
\\
2 & 4%
\end{array}%
\right ) \right ] \det \left [ \Sigma \left (
\begin{array}{c@{\quad }c}
2 & 4
\\
1 & 3%
\end{array}%
\right ) \right ] -\det \left [ \Sigma \left (
\begin{array}{c@{\quad }c}
1 & 3
\\
3 & 4%
\end{array}%
\right ) \right ] \det \left [ \Sigma \left (
\begin{array}{c@{\quad }c}
2 & 4
\\
1 & 2%
\end{array}%
\right ) \right ]
\end{eqnarray*}
\begin{equation*}
=-\det \left [
\begin{array}{c@{\quad }c}
K_{n+1}\left ( x,x\right ) & K_{n+1}\left ( x,y\right )
\\
K_{n+1}^{\left ( 0,1\right ) }\left ( x,x\right ) & K_{n+1}^{\left ( 0,1
\right ) }\left ( y,x\right )%
\end{array}%
\right ] \det \left [
\begin{array}{c@{\quad }c}
K_{n+1}^{\left ( 0,1\right ) }\left ( y,x\right ) & K_{n+1}^{\left ( 0,1
\right ) }\left ( y,y\right )
\\
K_{n+1}^{\left ( 1,1\right ) }\left ( x,y\right ) & K_{n+1}^{\left ( 1,1
\right ) }\left ( y,y\right )%
\end{array}%
\right ]
\end{equation*}
\begin{equation*}
-\det \left [
\begin{array}{c@{\quad }c}
K_{n+1}\left ( x,x\right ) & K_{n+1}^{\left ( 0,1\right ) }\left ( x,y
\right )
\\
K_{n+1}^{\left ( 0,1\right ) }\left ( x,x\right ) & K_{n+1}^{\left ( 1,1
\right ) }\left ( x,y\right )%
\end{array}%
\right ] \det \left [
\begin{array}{c@{\quad }c}
K_{n+1}\left ( y,y\right ) & K_{n+1}^{\left ( 0,1\right ) }\left ( y,x
\right )
\\
K_{n+1}^{\left ( 0,1\right ) }\left ( y,y\right ) & K_{n+1}^{\left ( 1,1
\right ) }\left ( x,y\right )%
\end{array}%
\right ]
\end{equation*}
\begin{equation*}
-\det \left [
\begin{array}{c@{\quad }c}
K_{n+1}\left ( x,y\right ) & K_{n+1}^{\left ( 0,1\right ) }\left ( x,x
\right )
\\
K_{n+1}^{\left ( 0,1\right ) }\left ( y,x\right ) & K_{n+1}^{\left ( 1,1
\right ) }\left ( x,x\right )%
\end{array}%
\right ] \det \left [
\begin{array}{c@{\quad }c}
K_{n+1}\left ( x,y\right ) & K_{n+1}^{\left ( 0,1\right ) }\left ( y,y
\right )
\\
K_{n+1}^{\left ( 0,1\right ) }\left ( x,y\right ) & K_{n+1}^{\left ( 1,1
\right ) }\left ( y,y\right )%
\end{array}%
\right ]
\end{equation*}
\begin{equation*}
+\det \left [
\begin{array}{c@{\quad }c}
K_{n+1}\left ( x,y\right ) & K_{n+1}^{\left ( 0,1\right ) }\left ( x,y
\right )
\\
K_{n+1}^{\left ( 0,1\right ) }\left ( y,x\right ) & K_{n+1}^{\left ( 1,1
\right ) }\left ( x,y\right )%
\end{array}%
\right ] \det \left [
\begin{array}{c@{\quad }c}
K_{n+1}\left ( x,y\right ) & K_{n+1}^{\left ( 0,1\right ) }\left ( y,x
\right )
\\
K_{n+1}^{\left ( 0,1\right ) }\left ( x,y\right ) & K_{n+1}^{\left ( 1,1
\right ) }\left ( x,y\right )%
\end{array}%
\right ]
\end{equation*}
\begin{equation*}
-\det \left [
\begin{array}{c@{\quad }c}
K_{n+1}^{\left ( 0,1\right ) }\left ( x,x\right ) & K_{n+1}^{\left ( 0,1
\right ) }\left ( x,y\right )
\\
K_{n+1}^{\left ( 1,1\right ) }\left ( x,x\right ) & K_{n+1}^{\left ( 1,1
\right ) }\left ( x,y\right )%
\end{array}%
\right ] \det \left [
\begin{array}{c@{\quad }c}
K_{n+1}\left ( x,y\right ) & K_{n+1}\left ( y,y\right )
\\
K_{n+1}^{\left ( 0,1\right ) }\left ( x,y\right ) & K_{n+1}^{\left ( 0,1
\right ) }\left ( y,y\right )%
\end{array}%
\right ] .
\end{equation*}%
Using the estimate {(\ref{eq3.7})} and that
$\left ( \left \vert x-y\right \vert +\frac{1%
}{n}\right ) ^{-1}\leq n$, we continue this as
\begin{eqnarray*}
&=&-\det \left [
\begin{array}{c@{\quad }c}
O\left ( n\right ) & O\left (
\frac{1}{\left \vert x-y\right \vert +\frac{1}{n}}%
\right )
\\\noalign{\vspace*{3pt}}
O\left ( n^{2}\right ) & O\left (
\frac{n}{\left \vert x-y\right \vert +\frac{1}{n%
}}\right )%
\end{array}%
\right ] \det \left [
\begin{array}{c@{\quad }c}
O\left ( \frac{n}{\left \vert x-y\right \vert +\frac{1}{n}}\right ) & O
\left ( n^{2}\right )
\\\noalign{\vspace*{3pt}}
O\left ( \frac{n^{2}}{\left \vert x-y\right \vert +\frac{1}{n}}
\right ) & O\left ( n^{3}\right )%
\end{array}%
\right ]
\\
&&{}-\det \left [
\begin{array}{c@{\quad }c}
O\left ( n\right ) & O\left (
\frac{n}{\left \vert x-y\right \vert +\frac{1}{n}}%
\right )
\\\noalign{\vspace*{3pt}}
O\left ( n^{2}\right ) & O\left (
\frac{n^{2}}{\left \vert x-y\right \vert +\frac{%
1}{n}}\right )%
\end{array}%
\right ] \det \left [
\begin{array}{c@{\quad }c}
O\left ( n\right ) & O\left (
\frac{n}{\left \vert x-y\right \vert +\frac{1}{n}}%
\right )
\\\noalign{\vspace*{3pt}}
O\left ( n^{2}\right ) & O\left (
\frac{n^{2}}{\left \vert x-y\right \vert +\frac{%
1}{n}}\right )%
\end{array}%
\right ]
\\
&&{}-\det \left [
\begin{array}{c@{\quad }c}
O\left ( \frac{1}{\left \vert x-y\right \vert +\frac{1}{n}}\right ) & O
\left ( n^{2}\right )
\\\noalign{\vspace*{3pt}}
O\left ( \frac{n}{\left \vert x-y\right \vert +\frac{1}{n}}\right ) & O
\left ( n^{3}\right )%
\end{array}%
\right ] \det \left [
\begin{array}{c@{\quad }c}
O\left ( \frac{1}{\left \vert x-y\right \vert +\frac{1}{n}}\right ) & O
\left ( n^{2}\right )
\\\noalign{\vspace*{3pt}}
O\left ( \frac{n}{\left \vert x-y\right \vert +\frac{1}{n}}\right ) & O
\left ( n^{3}\right )%
\end{array}%
\right ]
\\
&&{}
+\det \left [
\begin{array}{c@{\quad }c}
O\left ( \frac{1}{\left \vert x-y\right \vert +\frac{1}{n}}\right ) & O
\left ( \frac{n}{\left \vert x-y\right \vert +\frac{1}{n}}\right )
\\\noalign{\vspace*{3pt}}
O\left ( \frac{n}{\left \vert x-y\right \vert +\frac{1}{n}}\right ) & O
\left ( \frac{n^{2}}{\left \vert x-y\right \vert +\frac{1}{n}}\right )%
\end{array}%
\right ] \det \left [
\begin{array}{c@{\quad }c}
O\left ( \frac{1}{\left \vert x-y\right \vert +\frac{1}{n}}\right ) & O
\left ( \frac{n}{\left \vert x-y\right \vert +\frac{1}{n}}\right )
\\\noalign{\vspace*{3pt}}
O\left ( \frac{n}{\left \vert x-y\right \vert +\frac{1}{n}}\right ) & O
\left ( \frac{n^{2}}{\left \vert x-y\right \vert +\frac{1}{n}}\right )%
\end{array}%
\right ]
\\
&&{}-\det \left [
\begin{array}{c@{\quad }c}
O\left ( n^{2}\right ) & O\left (
\frac{n}{\left \vert x-y\right \vert +\frac{1}{n%
}}\right )
\\\noalign{\vspace*{3pt}}
O\left ( n^{3}\right ) & O\left (
\frac{n^{2}}{\left \vert x-y\right \vert +\frac{%
1}{n}}\right )%
\end{array}%
\right ] \det \left [
\begin{array}{c@{\quad }c}
O\left ( \frac{1}{\left \vert x-y\right \vert +\frac{1}{n}}\right ) & O
\left ( n\right )
\\\noalign{\vspace*{3pt}}
O\left ( \frac{n}{\left \vert x-y\right \vert +\frac{1}{n}}\right ) & O
\left ( n^{2}\right )%
\end{array}%
\right ]
\\
&=&O\left (
\frac{n^{6}}{\left ( \left \vert x-y\right \vert +\frac{1}{n}\right )
^{2}}\right ) .
\end{eqnarray*}
Thus%
%
%e4.6 #&#
\begin{equation}\label{eq4.6}
\text{Num}=O\left (
\frac{n^{6}}{\left ( \left \vert x-y\right \vert +\frac{1}{n}%
\right ) ^{2}}\right ) .
\end{equation}%
Also
\begin{eqnarray*}
\text{Denom} &=&\pi ^{2}\Delta \left [ \sqrt{\left ( \Omega _{11}
\Omega _{22}-\Omega _{12}^{2}\right ) \Delta }+\sqrt{\Psi \left ( x
\right ) \Psi \left ( y\right ) }\right ]
\\
&\geq &\pi ^{2}\Delta \sqrt{\Psi \left ( x\right ) \Psi \left ( y
\right ) }.
\end{eqnarray*}%
Here from {(\ref{eq3.14})} and {(\ref{eq3.13})}, for $n$ large enough,
\begin{equation*}
\Psi \left ( x\right ) =K_{n+1}^{\left ( 1,1\right ) }\left ( x,x
\right ) K_{n}\left ( x,x\right ) -K_{n}^{\left ( 0,1\right ) }\left (
x,x\right ) ^{2}\geq Cn^{4}-o\left ( n^{4}\right ) \geq Cn^{4}.
\end{equation*}%
Also from {(\ref{eq3.14})} and {(\ref{eq3.7})},
\begin{eqnarray*}
1-\frac{\Delta }{K_{n}\left ( x,x\right ) K_{n}\left ( y,y\right ) } &=&\frac{%
K_{n}^{2}\left ( x,y\right ) }{K_{n}\left ( x,x\right ) K_{n}\left ( y,y\right ) }
\\
&\leq &\frac{C}{\left ( \left \vert x-y\right \vert +\frac{1}{n}\right )
^{2}n^{2}}
\\
&=&\frac{C}{\left ( n\left \vert x-y\right \vert +1\right ) ^{2}}
\leq \frac{1}{2}%
,
\end{eqnarray*}%
if $\left \vert x-y\right \vert \geq \Lambda _{0}/n$ with
$\Lambda _{0}$ large enough. Then%
%
%e4.7 #&#
\begin{equation}\label{eq4.7}
\Delta \geq \frac{1}{2}K_{n}\left ( x,x\right ) K_{n}\left ( y,y
\right ) \geq Cn^{2}
\end{equation}%
and%
%
%e4.8 #&#
\begin{equation}\label{eq4.8}
\text{Denom}\geq Cn^{6}.
\end{equation}%
Then combined with {(\ref{eq4.6})}, this yields
\begin{equation*}
\left \vert T_{1}\right \vert =\left \vert \frac{\text{Num}}{\text{Denom}}%
\right \vert \leq
\frac{C}{\left ( \left \vert x-y\right \vert +\frac{1}{n}%
\right ) ^{2}}.\qedhere
\end{equation*}%
\end{proof}

Next, let us deal with $T_{2}$\textup{:}%

%l4.2 #&#
\begin{lem}\label{lem4.2}
There exist $n_{0}$ and $\Lambda _{0}$ such that
for $n\geq n_{0}$ and all $x,y\in \left [ a,b\right ]
$, with $\left \vert x-y\right \vert \geq \Lambda _{0}/n$,%
%
%e4.9 #&#
\begin{equation}\label{eq4.9}
\left \vert T_{2}\right \vert \leq
\frac{C}{\left ( \left \vert x-y\right \vert +%
\frac{1}{n}\right ) ^{2}}.
\end{equation}
\end{lem}

\begin{proof}
Recall that
\begin{equation*}
\left \vert T_{2}\right \vert =T_{2}=
\frac{1}{\pi ^{2}\sqrt{\Delta }}%
\left \vert \Omega _{12}\right \vert \arcsin \left (
\frac{\left \vert \Omega _{12}\right \vert }{\sqrt{\Omega _{11}\Omega _{22}}}
\right ) .
\end{equation*}%
From
$\left \vert \sin u\right \vert \geq \frac{2}{\pi }\left \vert u
\right \vert $, $\left \vert u\right \vert \leq \frac{\pi }{2}$, we obtain
for $\left \vert v\right \vert \leq 1$,
\begin{equation*}
\frac{2}{\pi }\left \vert \arcsin v\right \vert \leq \left \vert v
\right \vert \end{equation*}%
so%
%
%e4.10 #&#
\begin{equation}\label{eq4.10}
\left \vert T_{2}\right \vert \leq \frac{1}{2\pi \Delta ^{3/2}}
\frac{%
\left \vert \Omega _{12}\Delta \right \vert ^{2}}{\sqrt{\Omega _{11}\Omega _{22}\Delta ^{2}}}.
\end{equation}%
Here from {Lemma~\ref{lem3.1}}(d) and {Lemma~\ref{lem3.2}},
\begin{eqnarray*}
\Omega _{12}\Delta &=&\det \left [
\begin{array}{c@{\quad }c@{\quad }c}
K_{n+1}\left ( x,x\right ) & K_{n+1}\left ( x,y\right ) & K_{n+1}^{
\left ( 0,1\right ) }\left ( x,x\right )
\\
K_{n+1}\left ( y,x\right ) & K_{n+1}\left ( y,y\right ) & K_{n+1}^{
\left ( 0,1\right ) }\left ( y,x\right )
\\
K_{n+1}^{\left ( 1,0\right ) }\left ( y,x\right ) & K_{n+1}^{\left ( 0,1
\right ) }\left ( y,y\right ) & K_{n+1}^{\left ( 1,1\right ) }\left ( y,x
\right )%
\end{array}%
\right ]
\\
&=&\det \left [
\begin{array}{c@{\quad }c@{\quad }c}
O\left ( n\right ) & O\left (
\frac{1}{\left \vert x-y\right \vert +\frac{1}{n}}%
\right ) & O\left ( n^{2}\right )
\\
O\left ( \frac{1}{\left \vert x-y\right \vert +\frac{1}{n}}\right ) & O
\left ( n\right ) & O\left (
\frac{n}{\left \vert x-y\right \vert +\frac{1}{n}}\right )
\\
O\left ( \frac{n}{\left \vert x-y\right \vert +\frac{1}{n}}\right ) & O
\left ( n^{2}\right ) & O\left (
\frac{n^{2}}{\left \vert x-y\right \vert +\frac{1}{n}}%
\right )%
\end{array}%
\right ]
\end{eqnarray*}%
We expand by the first row and continue this as%
%
%e4.11 #&#
\begin{equation}\label{eq4.11}
\Omega _{12}\Delta =O\left (
\frac{n^{4}}{\left \vert x-y\right \vert +\frac{1}{%
n}}\right ) .
\end{equation}%
Next, we examine $\Omega _{11}$ and $\Omega _{22}$. From {Lemma~\ref{lem3.1}}(b),
followed by {(\ref{eq3.7})}, {(\ref{eq3.13})},
\begin{eqnarray*}
\Omega _{11}\Delta &=&\det \left [
\begin{array}{c@{\quad }c@{\quad }c}
K_{n+1}\left ( y,y\right ) & K_{n+1}\left ( y,x\right ) & K_{n+1}^{
\left ( 0,1\right ) }\left ( y,x\right )
\\
K_{n+1}\left ( x,y\right ) & K_{n+1}\left ( x,x\right ) & K_{n+1}^{
\left ( 0,1\right ) }\left ( x,x\right )
\\
K_{n+1}^{\left ( 1,0\right ) }\left ( x,y\right ) & K_{n+1}^{\left ( 0,1
\right ) }\left ( x,x\right ) & K_{n+1}^{\left ( 1,1\right ) }\left ( x,x
\right )%
\end{array}%
\right ]
\\
&=&\det \left [
\begin{array}{c@{\quad }c@{\quad }c}
K_{n+1}\left ( y,y\right ) & O\left (
\frac{1}{\left \vert x-y\right \vert +\frac{%
1}{n}}\right ) & O\left (
\frac{n}{\left \vert x-y\right \vert +\frac{1}{n}}%
\right )
\\\noalign{\vspace*{3pt}}
O\left ( \frac{1}{\left \vert x-y\right \vert +\frac{1}{n}}\right ) & K_{n+1}
\left ( x,x\right ) & o\left ( n^{2}\right )
\\\noalign{\vspace*{3pt}}
O\left ( \frac{n}{\left \vert x-y\right \vert +\frac{1}{n}}\right ) & o
\left ( n^{2}\right ) & K_{n+1}^{\left ( 1,1\right ) }\left ( x,x
\right )%
\end{array}%
\right ] .
\end{eqnarray*}%
Expanding by the first row, and using
$K_{n+1}^{\left ( r,r\right ) }\left ( x,x\right ) =O\left ( n^{2r+1}
\right ) $, we see that
\begin{eqnarray*}
\Omega _{11}\Delta &=&K_{n+1}\left ( y,y\right ) \left \{  K_{n+1}
\left ( x,x\right ) K_{n+1}^{\left ( 1,1\right ) }\left ( x,x\right ) -o
\left ( n^{4}\right ) \right \}
\\
&&-O\left ( \frac{1}{\left \vert x-y\right \vert +\frac{1}{n}}\right )
\left \{  O\left (
\frac{n^{3}}{\left \vert x-y\right \vert +\frac{1}{n}}\right ) +o
\left ( \frac{n^{3}}{\left \vert x-y\right \vert +\frac{1}{n}}\right )
\right \}
\\
&&+O\left ( \frac{n}{\left \vert x-y\right \vert +\frac{1}{n}}\right )
\left \{  O\left (
\frac{n^{2}}{\left \vert x-y\right \vert +\frac{1}{n}}\right ) +O
\left ( \frac{n^{2}}{\left \vert x-y\right \vert +\frac{1}{n}}\right )
\right \},
\end{eqnarray*}%
so if $\left \vert x-y\right \vert \geq \Lambda _{0}/n$, and
$\Lambda _{0}\geq 1$,%
%
%e4.12 #&#
\begin{eqnarray}\label{eq4.12}
\Omega _{11}\Delta &=&K_{n+1}\left ( y,y\right ) K_{n+1}\left ( x,x
\right ) K_{n+1}^{\left ( 1,1\right ) }\left ( x,x\right ) -o\left ( n^{5}
\right ) +O\left ( \frac{n^{5}}{\Lambda _{0}^{2}}\right )
\nonumber
\\
&\geq &Cn^{5}-o\left ( n^{5}\right ) +O\left (
\frac{n^{5}}{\Lambda _{0}^{2}}%
\right ) \geq C_{1}n^{5}
\end{eqnarray}%
if $\Lambda _{0}$ and $n$ are large enough, say $n\geq n_{0}$, by {(\ref{eq3.13})}
and {(\ref{eq3.14})}. Of course the constant $C_{1}$ depends on the size of
$C$, and the decay of the $o\left ( n^{5}\right ) $ term, as does
$n_{0}$. In much the same way,%
%
%e4.13 #&#
\begin{eqnarray}\label{eq4.13}
\Omega _{22}\Delta &=&\det \left [
\begin{array}{c@{\quad }c@{\quad }c}
K_{n+1}\left ( x,x\right ) & K_{n+1}\left ( x,y\right ) & K_{n+1}^{
\left ( 0,1\right ) }\left ( x,y\right )
\\
K_{n+1}\left ( y,x\right ) & K_{n+1}\left ( y,y\right ) & K_{n+1}^{
\left ( 0,1\right ) }\left ( y,y\right )
\\
K_{n+1}^{\left ( 1,0\right ) }\left ( y,x\right ) & K_{n+1}^{\left ( 1,0
\right ) }\left ( y,y\right ) & K_{n+1}^{\left ( 1,1\right ) }\left ( y,y
\right )%
\end{array}%
\right ]
\nonumber
\\
&=&K_{n+1}\left ( x,x\right ) K_{n+1}\left ( y,y\right ) K_{n+1}^{
\left ( 1,1\right ) }\left ( y,y\right ) -o\left ( n^{5}\right ) +O
\left ( \frac{n^{5}}{%
\Lambda _{0}^{2}}\right )
\nonumber
\\
&\geq &C_{1}n^{5}.
\end{eqnarray}%
Again the threshholds $n_{0}$ and $\Lambda _{0}$ influence the choice of
$%
C_{1}$. Then combining {(\ref{eq4.10})}--{(\ref{eq4.13})}, followed by {(\ref{eq4.7})},
\begin{equation*}
T_{2}\leq C\left (
\frac{n^{4}}{\left \vert x-y\right \vert +\frac{1}{n}}%
\right ) ^{2}\frac{1}{\Delta ^{3/2}}\frac{1}{n^{5}}\leq C\left (
\frac{1}{%
\left \vert x-y\right \vert +\frac{1}{n}}\right ) ^{2}.\qedhere
\end{equation*}%
\end{proof}

Next, we handle $T_{3}$\textup{:}%

%l4.3 #&#
\begin{lem}\label{lem4.3}
There exist $n_{0}$ and $\Lambda _{0}$ such that
for $n\geq n_{0}$ and all $x,y\in \left [ a,b\right ] $, with
$\left \vert x-y\right \vert \geq \Lambda _{0}/n$,
%
%e4.14 #&#
\begin{equation}\label{eq4.14}
\left \vert T_{3}\right \vert \leq
\frac{C}{\left ( \left \vert x-y\right \vert +%
\frac{1}{n}\right ) ^{2}}.
\end{equation}
\end{lem}
\begin{proof}
Note first from {(\ref{eq4.2})}, {(\ref{eq3.13})}, and {(\ref{eq3.14})},
\begin{equation*}
\Psi \left ( x\right ) =O\left ( n^{4}\right ) -o\left ( n^{4}\right )
=O\left ( n^{4}\right ) .
\end{equation*}%
Next, recall from {(\ref{eq4.4})},
\begin{eqnarray*}
T_{3} &=&\frac{1}{\pi ^{2}}\left ( \frac{1}{\Delta }-
\frac{1}{K_{n+1}\left (
x,x\right ) K_{n+1}\left ( y,y\right ) }\right ) \sqrt{\Psi \left ( x
\right ) \Psi \left ( y\right ) }
\\
&=&\frac{1}{\pi ^{2}}
\frac{K_{n+1}^{2}\left ( x,y\right ) }{\Delta K_{n+1}\left ( x,x\right ) K_{n+1}\left ( y,y\right ) }
\sqrt{\Psi \left ( x\right ) \Psi \left ( y\right ) }
\\
&\leq &\frac{C}{\left ( \left \vert x-y\right \vert +\frac{1}{n}\right )
^{2}\Delta n^{2}}n^{4}
\\
&\leq &\frac{C}{\left ( \left \vert x-y\right \vert +\frac{1}{n}\right ) ^{2}},
\end{eqnarray*}%
by {(\ref{eq4.7})}. Note too that $T_{3}\geq 0$.
\end{proof}

\begin{proof}[Proof of {Lemma~\ref{lem2.3}}(a)]
Just combine the estimates for $T_{1},T_{2},T_{3}$ from {Lemmas~\ref{lem4.1}, \ref{lem4.2},
\ref{lem4.3}} and recall {(\ref{eq4.3})}.
\end{proof}
\begin{proof}[Proof of {Lemma~\ref{lem2.3}}(b)]%
From {Lemma~\ref{lem2.3}}(a), for $y\in \left [ a,b\right ] $,
\begin{eqnarray*}
&&\int _{\left \{  x\in \left [ a,b\right ] ,\left \vert x-y\right
\vert \geq \Lambda /n\right \}  }\left \vert \rho _{2}\left ( x,y
\right ) -\rho _{1}\left ( x\right ) \rho _{1}\left ( y\right )
\right \vert dx
\\
&\leq &\int _{\left \{  x\in \left [ a,b\right ] ,\left \vert x-y
\right \vert \geq \Lambda /n\right \}  }
\frac{C}{\left \vert x-y\right \vert ^{2}}dx
\\
&\leq &\int _{\left \{  x\in \left [ a,b\right ] ,\left \vert x-y
\right \vert \geq \Lambda /n\right \}  }
\frac{2C}{\left \vert x-y\right \vert ^{2}+\left ( \frac{%
\Lambda }{n}\right ) ^{2}}dx
\\
&\leq &\int _{-\infty }^{\infty }
\frac{2C}{\left \vert x-y\right \vert ^{2}+\left ( \frac{\Lambda }{n}\right ) ^{2}}dx.
\end{eqnarray*}%
We make the substitution $x-y=\frac{\Lambda }{n}t$ in the integral:
\begin{equation*}
=\frac{n}{\Lambda }\int _{\mathbb{-\infty }}^{\infty }
\frac{2C}{t^{2}+1}dt.
\end{equation*}%
Then {(\ref{eq2.12})} follows.
\end{proof}

%s5 #&#
\section{The central term - {Lemma~\ref{lem2.4}}}
\label{sec5}

Recall that $\Delta ,\Omega _{11},\Omega _{22},\Omega _{12}$ were defined
in {(\ref{eq2.7})}--{(\ref{eq2.10})}, while $S,F,G,H$ were defined in {(\ref{eq1.3})}--{(\ref{eq1.6})}:%

%l5.1 #&#
\begin{lem}\label{lem5.1}
Uniformly for $u$ in compact subsets of the plane, and
uniformly for $x\in \left [ a,b\right ] $ and
$y=x+\frac{u}{n\omega \left ( x\right ) }$,%
\begin{enumerate}[(a)]
\item[\emph{(a)}]
%
%e5.1 #&#
\begin{equation}\label{eq5.1}
\frac{\left ( \Omega _{11}\Omega _{22}-\Omega _{12}^{2}\right ) \Delta }{%
K_{n+1}\left ( x,x\right ) ^{4}}\left (
\frac{1}{n\omega \left ( x\right ) }%
\right ) ^{4}=F\left ( u\right ) +o\left ( 1\right ) ;
\end{equation}%
\item[\emph{(b)}]
%
%e5.2 #&#
\begin{equation}\label{eq5.2}
\frac{\Delta }{K_{n+1}\left ( x,x\right ) ^{2}}=1-S\left ( u\right ) ^{2}+o
\left ( 1\right ) ;
\end{equation}%
\item[\emph{(c)}] %
%
%e5.3 #&#
\begin{equation}\label{eq5.3}
\frac{\Delta \Omega _{11}}{K_{n+1}\left ( x,x\right ) ^{3}}\left (
\frac{1}{%
n\omega \left ( x\right ) }\right ) ^{2}=G\left ( u\right ) +o\left ( 1
\right ) ;
\end{equation}%
\item[\emph{(d)}]
%
%e5.4 #&#
\begin{equation}\label{eq5.4}
\frac{\Delta \Omega _{22}}{K_{n+1}\left ( x,x\right ) ^{3}}\left (
\frac{1}{%
n\omega \left ( x\right ) }\right ) ^{2}=G\left ( u\right ) +o\left ( 1
\right ) ;
\end{equation}%
\item[\emph{(e)}]
%
%e5.5 #&#
\begin{equation}\label{eq5.5}
\frac{\Omega _{12}\Delta }{K_{n+1}\left ( x,x\right ) ^{3}}\left (
\frac{1}{%
n\omega \left ( x\right ) }\right ) ^{2}=H\left ( u\right ) +o\left ( 1
\right ) .
\end{equation}
\end{enumerate}
\end{lem}

\begin{proof}
We repeatedly use that
$\frac{K_{n+1}\left ( y,y\right ) }{K_{n+1}\left (
x,x\right ) }=1+o\left ( 1\right ) $, as follows from {(\ref{eq3.11})}.%

\noindent (a) Recall that $\Sigma $ was defined by {(\ref{eq3.1})}. Then {(\ref{eq3.6})} gives
\begin{eqnarray*}
&&\frac{\left [ \left ( \Omega _{11}\Omega _{22}-\Omega _{12}^{2}\right )
\Delta \right ] }{K_{n+1}\left ( x,x\right ) ^{4}}\left (
\frac{1}{n\omega \left ( x\right ) }\right ) ^{4}
\\
&=&\frac{\det \Sigma }{K_{n+1}\left ( x,x\right ) ^{4}}\left (
\frac{1}{n\omega \left ( x\right ) }\right ) ^{4}
\\
&=&\det \left[
\begin{array}{c@{\quad }c@{\quad }c@{\quad }c}
1 & \frac{K_{n+1}\left ( x,y\right ) }{K_{n+1}\left ( x,x\right ) } &
\frac{%
K_{n+1}^{\left ( 0,1\right ) }\left ( x,x\right ) }{K_{n+1}\left ( x,x\right ) }%
\frac{1}{n\omega \left ( x\right ) } &
\frac{K_{n+1}^{\left ( 0,1\right )
}\left ( x,y\right ) }{K_{n+1}\left ( x,x\right ) }
\frac{1}{n\omega \left (
x\right ) }
\\
\frac{K_{n+1}\left ( x,y\right ) }{K_{n+1}\left ( x,x\right ) } &
\frac{%
K_{n+1}\left ( y,y\right ) }{K_{n+1}\left ( x,x\right ) } &
\frac{%
K_{n+1}^{\left ( 0,1\right ) }\left ( y,x\right ) }{K_{n+1}\left ( x,x\right ) }%
\frac{1}{n\omega \left ( x\right ) } &
\frac{K_{n+1}^{\left ( 0,1\right )
}\left ( y,y\right ) }{K_{n+1}\left ( x,x\right ) }
\frac{1}{n\omega \left (
x\right ) }
\\
\frac{K_{n+1}^{\left ( 0,1\right ) }\left ( x,x\right ) }{K_{n+1}\left (
x,x\right ) }\frac{1}{n\omega \left ( x\right ) } &
\frac{K_{n+1}^{\left (
0,1\right ) }\left ( y,x\right ) }{K_{n+1}\left ( x,x\right ) }
\frac{1}{n\omega \left ( x\right ) } &
\frac{K_{n+1}^{\left ( 1,1\right ) }\left ( x,x\right ) }{%
K_{n+1}\left ( x,x\right ) }\left (
\frac{1}{n\omega \left ( x\right ) }\right ) ^{2} &
\frac{K_{n+1}^{\left ( 1,1\right ) }\left ( x,y\right ) }{K_{n+1}\left (
x,x\right ) }\left ( \frac{1}{n\omega \left ( x\right ) }\right ) ^{2}
\\
\frac{K_{n+1}^{\left ( 0,1\right ) }\left ( x,y\right ) }{K_{n+1}\left (
x,x\right ) }\frac{1}{n\omega \left ( x\right ) } &
\frac{K_{n+1}^{\left (
0,1\right ) }\left ( y,y\right ) }{K_{n+1}\left ( x,x\right ) }
\frac{1}{n\omega \left ( x\right ) } &
\frac{K_{n+1}^{\left ( 1,1\right ) }\left ( x,y\right ) }{%
K_{n+1}\left ( x,x\right ) }\left (
\frac{1}{n\omega \left ( x\right ) }\right ) ^{2} &
\frac{K_{n+1}^{\left ( 1,1\right ) }\left ( y,y\right ) }{K_{n+1}\left (
x,x\right ) }\left ( \frac{1}{n\omega \left ( x\right ) }\right ) ^{2}%
\end{array}%
\right ] .
\end{eqnarray*}
Here we have factored in $\frac{1}{n\omega \left ( x\right ) }$ into the
3rd and 4th rows and columns. Using {(\ref{eq3.9})} and recalling that
$y=x+\frac{u}{%
n\omega \left ( x\right ) }$, we continue this as
\begin{eqnarray*}
&=&\det \left [
\begin{array}{c@{\quad }c@{\quad }c@{\quad }c}
1 & S\left ( -u\right ) & -S^{\prime }\left ( 0\right ) & -S^{\prime }
\left ( -u\right )
\\
S\left ( -u\right ) & 1 & -S^{\prime }\left ( u\right ) & -S^{\prime }
\left ( 0\right )
\\
-S^{\prime }\left ( 0\right ) & -S^{\prime }\left ( u\right ) & -S^{
\prime\prime }\left ( 0\right ) & -S^{\prime\prime }\left ( -u
\right )
\\
-S^{\prime }\left ( -u\right ) & -S^{\prime }\left ( 0\right ) & -S^{
\prime\prime }\left ( -u\right ) & -S^{\prime\prime }\left ( 0
\right )%
\end{array}%
\right ] +o\left ( 1\right )
\\
&=&\det \left [
\begin{array}{c@{\quad }c@{\quad }c@{\quad }c}
1 & S\left ( u\right ) & 0 & S^{\prime }\left ( u\right )
\\
S\left ( u\right ) & 1 & -S^{\prime }\left ( u\right ) & 0
\\
0 & -S^{\prime }\left ( u\right ) & -S^{\prime\prime }\left ( 0
\right ) & -S^{\prime\prime }\left ( u\right )
\\
S^{\prime }\left ( u\right ) & 0 & -S^{\prime\prime }\left ( u
\right ) & -S^{\prime\prime }\left ( 0\right )%
\end{array}%
\right ] +o\left ( 1\right ) =F\left ( u\right ) +o\left ( 1\right )
\end{eqnarray*}
as $S$ is even, so $S^{\prime }$ is odd and $S^{\prime\prime }$ is even.%

\noindent (b) From {(\ref{eq3.9})},
\begin{equation*}
\frac{\Delta }{K_{n+1}\left ( x,x\right ) ^{2}}=\det \left [
\begin{array}{c@{\quad }c}
1 & \frac{K_{n+1}\left ( x,y\right ) }{K_{n+1}\left ( x,x\right ) }
\\
\frac{K_{n+1}\left ( x,y\right ) }{K_{n+1}\left ( x,x\right ) } &
\frac{%
K_{n+1}\left ( y,y\right ) }{K_{n+1}\left ( x,x\right ) }%
\end{array}%
\right ] =\det \left [
\begin{array}{c@{\quad }c}
1 & S\left ( -u\right )
\\
S\left ( -u\right ) & 1%
\end{array}%
\right ] +o\left ( 1\right ) .
\end{equation*}%
(c) From {(\ref{eq3.3})} and then {(\ref{eq3.9})},
\begin{eqnarray*}
&&\frac{\Delta \Omega _{11}}{K_{n+1}\left ( x,x\right ) ^{3}}\left (
\frac{1}{%
n\omega \left ( x\right ) }\right ) ^{2}
\\
&=&\det \left [
\begin{array}{c@{\quad }c@{\quad }c}
\frac{K_{n+1}\left ( y,y\right ) }{K_{n+1}\left ( x,x\right ) } &
\frac{%
K_{n+1}\left ( y,x\right ) }{K_{n+1}\left ( x,x\right ) } &
\frac{%
K_{n+1}^{\left ( 0,1\right ) }\left ( y,x\right ) }{K_{n+1}\left ( x,x\right ) }%
\frac{1}{n\omega \left ( x\right ) }
\\
\frac{K_{n+1}\left ( x,y\right ) }{K_{n+1}\left ( x,x\right ) } & 1 &
\frac{%
K_{n+1}^{\left ( 0,1\right ) }\left ( x,x\right ) }{K_{n+1}\left ( x,x\right ) }%
\frac{1}{n\omega \left ( x\right ) }
\\
\frac{K_{n+1}^{\left ( 1,0\right ) }\left ( x,y\right ) }{K_{n+1}\left (
x,x\right ) }\frac{1}{n\omega \left ( x\right ) } &
\frac{K_{n+1}^{\left (
0,1\right ) }\left ( x,x\right ) }{K_{n+1}\left ( x,x\right ) }
\frac{1}{n\omega \left ( x\right ) } &
\frac{K_{n+1}^{\left ( 1,1\right ) }\left ( x,x\right ) }{%
K_{n+1}\left ( x,x\right ) }\left (
\frac{1}{n\omega \left ( x\right ) }\right ) ^{2}%
\end{array}%
\right ]
\\
&=&\det \left [
\begin{array}{c@{\quad }c@{\quad }c}
1 & S\left ( u\right ) & -S^{\prime }\left ( u\right )
\\
S\left ( -u\right ) & 1 & -S^{\prime }\left ( 0\right )
\\
S^{\prime }\left ( -u\right ) & -S^{\prime }\left ( 0\right ) & -S^{
\prime\prime }\left ( 0\right )%
\end{array}%
\right ] +o\left ( 1\right )
\\
&=&\det \left [
\begin{array}{c@{\quad }c@{\quad }c}
1 & S\left ( u\right ) & -S^{\prime }\left ( u\right )
\\
S\left ( u\right ) & 1 & 0
\\
-S^{\prime }\left ( u\right ) & 0 & -S^{\prime\prime }\left ( 0
\right )%
\end{array}%
\right ] +o\left ( 1\right ) =G\left ( u\right ) +o\left ( 1\right ) ,
\end{eqnarray*}%
recall {(\ref{eq1.5})}.%

\noindent (d) From {(\ref{eq3.4})} and then {(\ref{eq3.9})},
\begin{eqnarray*}
&&\frac{\Delta \Omega _{22}}{K_{n+1}\left ( x,x\right ) ^{3}}\left (
\frac{1}{%
n\omega \left ( x\right ) }\right ) ^{2}
\\
&=&\det \left [
\begin{array}{c@{\quad }c@{\quad }c}
1 & \frac{K_{n+1}\left ( x,y\right ) }{K_{n+1}\left ( x,x\right ) } &
\frac{%
K_{n+1}^{\left ( 0,1\right ) }\left ( x,y\right ) }{K_{n+1}\left ( x,x\right ) }%
\frac{1}{n\omega \left ( x\right ) }
\\
\frac{K_{n+1}\left ( y,x\right ) }{K_{n+1}\left ( x,x\right ) } &
\frac{%
K_{n+1}\left ( y,y\right ) }{K_{n+1}\left ( x,x\right ) } &
\frac{%
K_{n+1}^{\left ( 0,1\right ) }\left ( y,y\right ) }{K_{n+1}\left ( x,x\right ) }%
\frac{1}{n\omega \left ( x\right ) }
\\
\frac{K_{n+1}^{\left ( 1,0\right ) }\left ( y,x\right ) }{K_{n+1}\left (
x,x\right ) }\frac{1}{n\omega \left ( x\right ) } &
\frac{K_{n+1}^{\left (
1,0\right ) }\left ( y,y\right ) }{K_{n+1}\left ( x,x\right ) }
\frac{1}{n\omega \left ( x\right ) } &
\frac{K_{n+1}^{\left ( 1,1\right ) }\left ( y,y\right ) }{%
K_{n+1}\left ( x,x\right ) }\left (
\frac{1}{n\omega \left ( x\right ) }\right ) ^{2}%
\end{array}%
\right ]
\\
&=&\det \left [
\begin{array}{c@{\quad }c@{\quad }c}
1 & S\left ( -u\right ) & -S^{\prime }\left ( -u\right )
\\
S\left ( u\right ) & 1 & -S^{\prime }\left ( 0\right )
\\
S^{\prime }\left ( u\right ) & S^{\prime }\left ( 0\right ) & -S^{
\prime\prime }\left ( 0\right )%
\end{array}%
\right ] +o\left ( 1\right ) =G\left ( u\right ) +o\left ( 1\right )
\end{eqnarray*}%
as $S^{\prime }$ is odd, and we can multiply both the 3rd row and 3rd column
by $-1$.%

\noindent (e) From {(\ref{eq3.5})} and then {(\ref{eq3.9})},
\begin{eqnarray*}
&&\frac{\Omega _{12}\Delta }{K_{n+1}\left ( x,x\right ) ^{3}}\left (
\frac{1}{%
n\omega \left ( x\right ) }\right ) ^{2}
\\
&=&\det \left [
\begin{array}{c@{\quad }c@{\quad }c}
1 & \frac{K_{n+1}\left ( x,y\right ) }{K_{n+1}\left ( x,x\right ) } &
\frac{%
K_{n+1}^{\left ( 0,1\right ) }\left ( x,x\right ) }{K_{n+1}\left ( x,x\right ) }%
\frac{1}{n\omega \left ( x\right ) }
\\
\frac{K_{n+1}\left ( y,x\right ) }{K_{n+1}\left ( x,x\right ) } &
\frac{%
K_{n+1}\left ( y,y\right ) }{K_{n+1}\left ( x,x\right ) } &
\frac{%
K_{n+1}^{\left ( 0,1\right ) }\left ( y,x\right ) }{K_{n+1}\left ( x,x\right ) }%
\frac{1}{n\omega \left ( x\right ) }
\\
\frac{K_{n+1}^{\left ( 1,0\right ) }\left ( y,x\right ) }{K_{n+1}\left (
x,x\right ) }\frac{1}{n\omega \left ( x\right ) } &
\frac{K_{n+1}^{\left (
0,1\right ) }\left ( y,y\right ) }{K_{n+1}\left ( x,x\right ) }
\frac{1}{n\omega \left ( x\right ) } &
\frac{K_{n+1}^{\left ( 1,1\right ) }\left ( y,x\right ) }{%
K_{n+1}\left ( x,x\right ) }\left (
\frac{1}{n\omega \left ( x\right ) }\right ) ^{2}%
\end{array}%
\right ]
\\
&=&\det \left [
\begin{array}{c@{\quad }c@{\quad }c}
1 & S\left ( -u\right ) & 0
\\
S\left ( u\right ) & 1 & -S^{\prime }\left ( u\right )
\\
S^{\prime }\left ( u\right ) & 0 & -S^{\prime\prime }\left ( u
\right )%
\end{array}%
\right ] +o\left ( 1\right ) =H\left ( u\right ) +o\left ( 1\right ) ,
\end{eqnarray*}%
recall {(\ref{eq1.6})}.
\end{proof}

Now we can obtain the asymptotics for
$\rho _{2}\left ( x,y\right ) -\rho _{1}\left ( x\right ) \rho _{1}
\left ( y\right ) $ stated in {(\ref{eq2.13})}:%

\begin{proof}[Proof of {Lemma~\ref{lem2.4}}(a)]%
Recall as in {(\ref{eq4.3})}--{(\ref{eq4.4})}, that
%
%e5.6 #&#
\begin{eqnarray}\label{eq5.6}
&&\left ( \frac{1}{n\omega \left ( x\right ) }\right ) ^{2}\left \{
\rho _{2}\left ( x,y\right ) -\rho _{1}\left ( x\right ) \rho _{1}
\left ( y\right ) \right \}
\nonumber
\\
&=&\left ( \frac{1}{n\omega \left ( x\right ) }\right ) ^{2}\left \{  T_{1}+T_{2}+T_{3}
\right \}  .
\end{eqnarray}%
We handle the terms $T_{j},j=1,2,3$ one by one:%

\noindent \textbf{Step 1:} $T_{1}$%

\noindent Firstly from {(\ref{eq3.9})}, {(\ref{eq3.10})}, {(\ref{eq3.15})}, and {(\ref{eq4.2})},
\begin{eqnarray*}
&&\frac{\Psi \left ( x\right ) }{K_{n+1}\left ( x,x\right ) ^{2}}
\left ( \frac{1}{%
n\omega \left ( x\right ) }\right ) ^{2}
\\
&=&\left (
\frac{K_{n+1}^{\left ( 1,1\right ) }\left ( x,x\right ) }{%
K_{n+1}\left ( x,x\right ) }-\left (
\frac{K_{n+1}^{\left ( 0,1\right ) }\left (
x,x\right ) }{K_{n+1}\left ( x,x\right ) }\right ) ^{2}\right )
\left ( \frac{1}{%
n\omega \left ( x\right ) }\right ) ^{2}
\\
&=&-S^{\prime\prime }\left ( 0\right ) +o\left ( 1\right ) =
\frac{\pi ^{2}}{3}%
+o\left ( 1\right ) .
\end{eqnarray*}%
Also, from {(\ref{eq3.9})}, uniformly for $u$ in compact subsets of
$\mathbb{C}$,
%
%e5.7 #&#
\begin{eqnarray}\label{eq5.7}
&&\frac{\Psi \left ( y\right ) }{K_{n+1}\left ( x,x\right ) ^{2}}
\left ( \frac{1}{%
n\omega \left ( x\right ) }\right ) ^{2}
\nonumber
\\
&=&\left (
\frac{K_{n+1}^{\left ( 1,1\right ) }\left ( y,y\right ) }{%
K_{n+1}\left ( x,x\right ) }
\frac{K_{n+1}\left ( y,y\right ) }{K_{n+1}\left (
x,x\right ) }-\left (
\frac{K_{n+1}^{\left ( 0,1\right ) }\left ( y,y\right ) }{%
K_{n+1}\left ( x,x\right ) }\right ) ^{2}\right ) \left (
\frac{1}{n\omega \left (
x\right ) }\right ) ^{2}
\nonumber
\\
&=&-S^{\prime\prime }\left ( 0\right ) +o\left ( 1\right ) =
\frac{\pi ^{2}}{3}%
+o\left ( 1\right ) .
\end{eqnarray}%
Then
\begin{eqnarray*}
&&\left ( \frac{1}{n\omega \left ( x\right ) }\right ) ^{4}
\frac{\Psi \left (
x\right ) \Psi \left ( y\right ) }{\Delta ^{2}}
\\
&=&\frac{K_{n+1}\left ( x,x\right ) ^{4}}{\Delta ^{2}}\left [
\frac{\Psi \left (
x\right ) }{K_{n+1}\left ( x,x\right ) ^{2}}\left (
\frac{1}{n\omega \left (
x\right ) }\right ) ^{2}\right ] \left [
\frac{\Psi \left ( y\right ) }{%
K_{n+1}\left ( x,x\right ) ^{2}}\left (
\frac{1}{n\omega \left ( x\right ) }%
\right ) ^{2}\right ]
\\
&=&\frac{1}{\left ( 1-S\left ( u\right ) ^{2}\right ) ^{2}}\left (
\frac{\pi ^{2}%
}{3}\right ) ^{2}+o\left ( 1\right ) ,
\end{eqnarray*}%
by the above and {Lemma~\ref{lem5.1}}(b). Hence also with an obvious choice of branches,
uniformly for $u$ in compact subsets of
$\mathbb{C}\backslash \left \{  0\right \}  $,
%
%e5.8 #&#
\begin{equation}\label{eq5.8}
\left ( \frac{1}{n\omega \left ( x\right ) }\right ) ^{2}
\frac{1}{\Delta }\sqrt{%
\Psi \left ( x\right ) \Psi \left ( y\right ) }=
\frac{1}{1-S\left ( u\right ) ^{2}}%
\left ( \frac{\pi ^{2}}{3}\right ) +o\left ( 1\right ) .
\end{equation}%
(Note that $\Delta $ occurs outside the square root, and only it leads
to the pole at $0$). Then from {(\ref{eq5.1})} and {(\ref{eq5.8})}, and recalling the definition
of $T_{1}$ at {(\ref{eq4.4})},
\begin{eqnarray*}
&&\left ( \frac{1}{n\omega \left ( x\right ) }\right ) ^{2}T_{1}
\\
&=&\frac{K_{n+1}\left ( x,x\right ) ^{2}}{\pi ^{2}\Delta }\sqrt{
\frac{\left (
\Omega _{11}\Omega _{22}-\Omega _{12}^{2}\right ) \Delta }{K_{n+1}\left (
x,x\right ) ^{4}\left ( n\omega \left ( x\right ) \right ) ^{4}}}-
\left ( \frac{1}{%
n\omega \left ( x\right ) }\right ) ^{2}\frac{1}{\pi ^{2}\Delta }
\sqrt{\Psi \left ( x\right ) \Psi \left ( y\right ) }
\\
&=&\frac{1}{\pi ^{2}\left ( 1-S\left ( u\right ) ^{2}\right ) }\left (
\sqrt{%
F\left ( u\right ) }-\frac{\pi ^{2}}{3}\right ) +o\left ( 1\right ) .
\end{eqnarray*}%
\textbf{Step 2:} $T_{2}$%

\noindent From {(\ref{eq4.4})},
\begin{eqnarray*}
&&\ \left ( \frac{1}{n\omega \left ( x\right ) }\right ) ^{2}T_{2}
\\
&=&\frac{K_{n+1}\left ( x,x\right ) ^{3}}{\pi ^{2}\Delta ^{3/2}}
\left [ \frac{%
\Omega _{12}\Delta }{K_{n+1}\left ( x,x\right ) ^{3}}\left (
\frac{1}{n\omega \left ( x\right ) }\right ) ^{2}\right ] \arcsin
\left ( \frac{\Omega _{12}}{%
\sqrt{\Omega _{11}\Omega _{22}}}\right )
\\
&=&\frac{1}{\pi ^{2}\left ( 1-S\left ( u\right ) ^{2}\right ) ^{3/2}}H
\left ( u\right ) \arcsin \left (
\frac{H\left ( u\right ) }{G\left ( u\right ) }\right ) +o\left ( 1
\right ) ,
\end{eqnarray*}%
by {(\ref{eq5.2})}--{(\ref{eq5.5})}.%

\noindent \textbf{Step 3: }$T_{3}$%

\noindent From {(\ref{eq4.4})} and {(\ref{eq5.5})},
\begin{eqnarray*}
&&\left ( \frac{1}{n\omega \left ( x\right ) }\right ) ^{2}T_{3}
\\
&=&\left ( \frac{1}{n\omega \left ( x\right ) }\right ) ^{2}
\frac{1}{\pi ^{2}}%
\left (
\frac{K_{n+1}\left ( x,y\right ) ^{2}}{\Delta K_{n+1}\left ( x,x\right )
K_{n+1}\left ( y,y\right ) }\right ) \sqrt{\Psi \left ( x\right )
\Psi \left ( y\right ) }
\\
&=&\frac{1}{\pi ^{2}}\left (
\frac{K_{n+1}\left ( x,y\right ) }{K_{n+1}\left (
x,x\right ) }\right ) ^{2}
\frac{K_{n+1}\left ( x,x\right ) }{K_{n+1}\left (
y,y\right ) }\left [
\frac{1}{\left ( n\omega \left ( x\right ) \right )
^{2}\Delta }\sqrt{\Psi \left ( x\right ) \Psi \left ( y\right ) }
\right ]
\\
&=&\frac{1}{\pi ^{2}}\left (
\frac{S\left ( u\right ) ^{2}}{1-S\left ( u\right )
^{2}}\right ) \frac{\pi ^{2}}{3}+o\left ( 1\right ) ,
\end{eqnarray*}%
by {(\ref{eq5.8})} and {(\ref{eq3.9})}. Substituting the asymptotics for $T_{j},j=1,2,3$ into
{(\ref{eq5.6})} gives
\begin{eqnarray*}
&&\left ( \frac{1}{n\omega \left ( x\right ) }\right ) ^{2}\left \{
\rho _{2}\left ( x,y\right ) -\rho _{1}\left ( x\right ) \rho _{1}
\left ( y\right ) \right \}
\\
&=&\frac{1}{\pi ^{2}\left ( 1-S\left ( u\right ) ^{2}\right ) }\left
\{  \sqrt{%
F\left ( u\right ) }-\frac{\pi ^{2}}{3}\left ( 1-S\left ( u\right ) ^{2}
\right ) +%
\frac{H\left ( u\right ) }{\sqrt{1-S\left ( u\right ) ^{2}}}\arcsin
\left ( \frac{%
H\left ( u\right ) }{G\left ( u\right ) }\right ) \right \}  +o\left ( 1
\right )
\\
&=&\Xi \left ( u\right ) +o\left ( 1\right ) ,
\end{eqnarray*}%
recall {(\ref{eq1.7})}.
\end{proof}

We next deal with $u$ near $0$, which turns out to be challenging. First,
we prove%

%l5.2 #&#
\begin{lem}\label{lem5.2}
\begin{enumerate}[(a)]
\item[\emph{(a)}] $\Delta \left ( x,x+\frac{u}{n\omega \left ( x\right ) }
\right ) $
 has a double zero at $u=0$, and there is $\rho >0$
 such that for all $x\in \left [ a,b\right ] $ and $n$
 large enough, $\Delta \left ( x,x+
\frac{u}{n\omega \left ( x\right )
}\right ) $ has no other zeros in
$\left \vert u\right \vert \leq \rho $.
Moreover, uniformly for $u$ in compact subsets of
 $\mathbb{C}$, and $x\in \left [ a,b\right ] $,
 %
%e5.9 #&#
\begin{equation}\label{eq5.9}
\lim _{n\rightarrow \infty }
\frac{\Delta \left ( x,x+\frac{u}{n\omega \left (
x\right ) }\right ) }{K_{n+1}\left ( x,x\right ) ^{2}u^{2}}=
\frac{1-S\left (
u\right ) ^{2}}{u^{2}}.
\end{equation}%
The right-hand side is interpreted as its limiting value at $u=0$.

\item[\emph{(b)}] $\left [ \left ( \Omega _{11}\Omega _{22}-\Omega _{12}^{2}
\right ) \Delta \right ] \left ( x,x+
\frac{u}{n\omega \left ( x\right ) }\right ) $
has a zero of even order at least $4$ at $u=0$.
Moreover, uniformly for $u$ in compact subsets of
$\mathbb{C}$, and $x\in \left [ a,b\right ] $,
\begin{equation*}
\lim _{n\rightarrow \infty }
\frac{\left ( \Omega _{11}\Omega _{22}-\Omega _{12}^{2}\right ) }{\Delta }
\left ( \frac{1}{n\omega \left ( x\right ) }\right ) ^{4}=
\frac{F\left ( u\right ) }{\left ( 1-S\left ( u\right ) ^{2}\right ) ^{2}}.
\end{equation*}%
The right-hand side is interpreted as its limiting value at
$u=0$.
\end{enumerate}
\end{lem}
\begin{proof}
(a) First,
\begin{eqnarray*}
&&\Delta \left ( x,x+\frac{u}{n\omega \left ( x\right ) }\right )
\\
&=&K_{n+1}\left ( x,x\right ) K_{n+1}\left ( x+
\frac{u}{n\omega \left ( x\right )
},x+\frac{u}{n\omega \left ( x\right ) }\right ) -K_{n+1}\left ( x,x+
\frac{u}{%
n\omega \left ( x\right ) }\right ) ^{2}
\end{eqnarray*}%
is a polynomial in $u$, and by Cauchy-Schwarz is non-negative for real
$u$, with a zero at $u=0$. This then must be a zero of even multiplicity.
But since
\begin{equation*}
\lim _{n\rightarrow \infty }
\frac{\Delta \left ( x,x+\frac{u}{n\omega \left (
x\right ) }\right ) }{K_{n+1}\left ( x,x\right ) ^{2}}=1-S\left ( u
\right ) ^{2},
\end{equation*}%
uniformly in compact sets by {Lemma~\ref{lem5.1}}(b) and {(\ref{eq3.9})}, and the right-hand
side has an isolated double zero at $0$, it follows from Hurwitz' Theorem
and the considerations above, that necessarily for large enough $n$,
$\Delta \left ( x,x+\frac{u}{n\omega \left ( x\right ) }\right ) $ has
a double zero at $0$, and no other zeros in some neighborhood of $0$ that
is independent of $n$. Since the convergence is uniform in $x$, the neighborhood
may also be taken independent of $x$. But then
$\left \{
\frac{\Delta \left ( x,x+\frac{u}{%
n\omega \left ( x\right ) }\right ) }{K_{n+1}\left ( x,x\right ) ^{2}u^{2}}%
\right \}  _{n\geq 1}$ is a sequence of polynomials in $u$ that converges
uniformly in compact subsets of
$\mathbb{C}\backslash \left \{  0\right \}  $ and hence also in compact subsets
of $\mathbb{C}$.%

\noindent (b) Recall {(\ref{eq3.6})}:
\begin{equation*}
\left ( \Omega _{11}\Omega _{22}-\Omega _{12}^{2}\right ) \Delta =
\det \left ( \Sigma \right ) .
\end{equation*}%
Here $\mathrm{det}\left ( \Sigma \right ) $ is also a polynomial in $u$ when
$y=x+\frac{u}{n\omega \left ( x\right ) }$. As in the proof of {Lemma~\ref{lem2.2}} in the Appendix,
$\Sigma $ is a positive definite matrix when $x\neq y$, so is nonegative
definite for all $x,y$. Then $\mathrm{det}\left ( \Sigma \right ) \geq 0$ for real
$x,y$ while $\det \left ( \Sigma \right ) =0$ when $u=0$. Thus as a polynomial
in $u$, $\mathrm{det}\left ( \Sigma \right ) $ can only have an even multiplicity zero
at $u=0 $. We need to show that it has a zero of multiplicity at least
$4$ when $u=0$. By a classical inequality for determinants of positive definite matrices
and their leading submatrices
\cite[p. 63, Thm. 7]{BeckenbachBellman1961}, when $y$ is real,
\begin{equation*}
0\leq \det \left ( \Sigma \right ) \leq \Delta \left ( x,y\right )
\det \left [
\begin{array}{c@{\quad }c}
K_{n+1}^{\left ( 1,1\right ) }\left ( x,x\right ) & K_{n+1}^{\left ( 1,1
\right ) }\left ( x,y\right )
\\
K_{n+1}^{\left ( 1,1\right ) }\left ( x,y\right ) & K_{n+1}^{\left ( 1,1
\right ) }\left ( y,y\right )%
\end{array}%
\right ] .
\end{equation*}%
We already know that $\Delta $ has a double zero at $u=0$ for
$y=x+\frac{u}{%
n\omega \left ( x\right ) }$. But the second determinant also vanishes
when $%
y=x$, that is $u=0$. It follows that necessarily as a polynomial in
$u$, $%
\det \left ( \Sigma \right ) $ has a zero of multiplicity at least
$4$ at $u=0$. Then
\begin{equation*}
\frac{\Omega _{11}\Omega _{22}-\Omega _{12}^{2}}{\Delta }=
\frac{\det \left (
\Sigma \right ) }{\Delta ^{2}}
\end{equation*}%
has a removable singularity at $0$, since the zero of multiplicity
$4$ in the denominator is cancelled by the zero of multiplicity
$\geq 4$ in the numerator. Then from {(\ref{eq5.1})}, {(\ref{eq5.2})}, uniformly for
$x\in \left [ a,b\right ] $ and $u$ in some neighborhood of $0$,
\begin{eqnarray*}
&&\frac{\Omega _{11}\Omega _{22}-\Omega _{12}^{2}}{\Delta }\left (
\frac{1}{%
n\omega \left ( x\right ) }\right ) ^{4}
\\
&=&\frac{\left ( \Omega _{11}\Omega _{22}-\Omega _{12}^{2}\right ) \Delta }{%
K_{n+1}\left ( x,x\right ) ^{4}}\left (
\frac{1}{n\omega \left ( x\right ) }%
\right ) ^{4}\left [ \frac{K_{n+1}\left ( x,x\right ) ^{2}}{\Delta }
\right ] ^{2}
\\
&=&\frac{F\left ( u\right ) }{\left ( 1-S\left ( u\right ) ^{2}\right ) ^{2}}%
+o\left ( 1\right ) .
\end{eqnarray*}%
Moreover, since $S\left ( u\right ) =1$ only at $u=0$, this limit actually
holds uniformly for $u$ in compact subsets of $\mathbb{C}$.
\end{proof}

Next, we deal with the most difficult term $\Omega _{12}$:%

%l5.3 #&#
\begin{lem}\label{lem5.3}
$\left ( \Omega _{12}\Delta \right ) \left ( x,x+
\frac{u}{n\omega \left (
x\right ) }\right ) $ has a zero of multiplicity at least $3$ at
$u=0$. Moreover, uniformly for $u$ in compact subsets
of $\mathbb{R}$, and $x\in \left [ a,b\right ] $,
\begin{equation*}
\lim _{n\rightarrow \infty }\frac{\Omega _{12}}{\sqrt{\Delta }}\left (
\frac{1%
}{n\omega \left ( x\right ) }\right ) ^{2}=
\frac{H\left ( u\right ) }{\left (
1-S^{2}\left ( u\right ) \right ) ^{3/2}}.
\end{equation*}%
The right-hand side is interpreted as its limiting value at $u=0$.
In addition, uniformly for $u$ in compact subsets of
$\mathbb{R}$, and $x\in \left [ a,b\right ] $,
\begin{equation*}
\frac{\left \vert \Omega _{12}\right \vert }{\sqrt{\Delta }}\arcsin
\left (
\frac{\left \vert \Omega _{12}\right \vert }{\sqrt{\Omega _{11}\Omega _{22}}}%
\right ) \left ( \frac{1}{n\omega \left ( x\right ) }\right ) ^{2}
\leq C.
\end{equation*}%
\end{lem}

\begin{proof}
We first perform row and column operations in the determinant defining
$\Delta _{12}$ and then expand using Taylor series. More precisely, we subtract
the first row from the second; then the first column from the second; and then we subtract $\frac{1}{y-x}$ $\times $ the second row from the third:
%
%e5.10 #&#
%e5.11 #&#
%e5.12 #&#
\begin{eqnarray}\label{eq5.10}
&&\Omega _{12}\Delta
\nonumber
\\
&=&\det \left [
\begin{array}{c@{\quad }c@{\quad }c}
K_{n+1}\left ( x,x\right ) & K_{n+1}\left ( x,y\right ) & K_{n+1}^{
\left ( 0,1\right ) }\left ( x,x\right )
\\
K_{n+1}\left ( y,x\right ) & K_{n+1}\left ( y,y\right ) & K_{n+1}^{
\left ( 0,1\right ) }\left ( y,x\right )
\\
K_{n+1}^{\left ( 1,0\right ) }\left ( y,x\right ) & K_{n+1}^{\left ( 0,1
\right ) }\left ( y,y\right ) & K_{n+1}^{\left ( 1,1\right ) }\left ( y,x
\right )%
\end{array}%
\right ]
\nonumber
\\
&=&\det \left [
\begin{array}{c@{\quad }c@{\quad }c}
K_{n+1}\left ( x,x\right ) & K_{n+1}\left ( x,y\right ) & K_{n+1}^{
\left ( 0,1\right ) }\left ( x,x\right )
\\
K_{n+1}\left ( y,x\right ) -K_{n+1}\left ( x,x\right ) & K_{n+1}
\left ( y,y\right ) -K_{n+1}\left ( x,y\right ) & K_{n+1}^{\left ( 0,1
\right ) }\left ( y,x\right ) -K_{n+1}^{\left ( 0,1\right ) }\left ( x,x
\right )
\\
K_{n+1}^{\left ( 1,0\right ) }\left ( y,x\right ) & K_{n+1}^{\left ( 0,1
\right ) }\left ( y,y\right ) & K_{n+1}^{\left ( 1,1\right ) }\left ( y,x
\right )%
\end{array}%
\right ]
\nonumber
\\
&=&\det \left [
\begin{array}{c@{\quad }c@{\quad }c}
K_{n+1}\left ( x,x\right ) & K_{n+1}\left ( x,y\right ) -K_{n+1}
\left ( x,x\right ) & K_{n+1}^{\left ( 0,1\right ) }\left ( x,x
\right )
\\
K_{n+1}\left ( y,x\right ) -K_{n+1}\left ( x,x\right ) &
\begin{array}{c}
K_{n+1}\left ( y,y\right ) +K_{n+1}\left ( x,x\right )
\\
-2K_{n+1}\left ( x,y\right )%
\end{array}
& K_{n+1}^{\left ( 0,1\right ) }\left ( y,x\right ) -K_{n+1}^{\left ( 0,1
\right ) }\left ( x,x\right )
\\
K_{n+1}^{\left ( 1,0\right ) }\left ( y,x\right ) & K_{n+1}^{\left ( 1,0
\right ) }\left ( y,y\right ) -K_{n+1}^{\left ( 1,0\right ) }\left ( y,x
\right ) & K_{n+1}^{\left ( 1,1\right ) }\left ( y,x\right )%
\end{array}%
\right ]
\nonumber
\\
&=&\det \left [
\begin{array}{c@{\quad }c@{\quad }c}
K_{n+1}\left ( x,x\right ) & K_{n+1}\left ( x,y\right ) -K_{n+1}
\left ( x,x\right ) & K_{n+1}^{\left ( 0,1\right ) }\left ( x,x
\right )
\\
K_{n+1}\left ( y,x\right ) -K_{n+1}\left ( x,x\right ) &
\begin{array}{c}
K_{n+1}\left ( y,y\right ) +K_{n+1}\left ( x,x\right )
\\
-2K_{n+1}\left ( x,y\right )%
\end{array}
& K_{n+1}^{\left ( 0,1\right ) }\left ( y,x\right ) -K_{n+1}^{\left ( 0,1
\right ) }\left ( x,x\right )
\\
\begin{array}{c}
K_{n+1}^{\left ( 1,0\right ) }\left ( y,x\right )
\\
-
\frac{K_{n+1}\left ( y,x\right ) -K_{n+1}\left ( x,x\right ) }{y-x}%
\end{array}
&
\begin{array}{c}
K_{n+1}^{\left ( 1,0\right ) }\left ( y,y\right ) -K_{n+1}^{\left ( 1,0
\right ) }\left ( y,x\right )
\\
-
\frac{K_{n+1}\left ( y,y\right ) +K_{n+1}\left ( x,x\right ) -2K_{n+1}\left (
x,y\right ) }{y-x}%
\end{array}
&
\begin{array}{c}
K_{n+1}^{\left ( 1,1\right ) }\left ( y,x\right )
\\
-
\frac{K_{n+1}^{\left ( 0,1\right ) }\left ( y,x\right ) -K_{n+1}^{\left (
0,1\right ) }\left ( x,x\right ) }{y-x}%
\end{array}%
\end{array}%
\right ].
\nonumber
\\
\end{eqnarray}
Let us examine the entries in the second and third rows. First, for some
$t$ between $x,y$,
\begin{equation*}
K_{n+1}\left ( y,x\right ) -K_{n+1}\left ( x,x\right ) =K_{n+1}^{
\left ( 1,0\right ) }\left ( t,x\right ) \left ( y-x\right ) =O\left (
n^{2}\left ( y-x\right ) \right )
\end{equation*}%
by {Lemma~\ref{lem3.2}}. Second, using the estimates from that lemma, for some
$r,s,v$ between $x,y$,
\begin{eqnarray*}
&&K_{n+1}\left ( y,y\right ) +K_{n+1}\left ( x,x\right ) -2K_{n+1}
\left ( x,y\right )
\\
&=&K_{n+1}\left ( x,x\right ) +\left ( y-x\right ) 2K_{n+1}^{\left ( 1,0
\right ) }\left ( x,x\right ) +\frac{1}{2}\left ( y-x\right ) ^{2}2
\left \{  K_{n+1}^{\left ( 1,1\right ) }\left ( r,r\right ) +K_{n+1}^{
\left ( 2,0\right ) }\left ( r,r\right ) \right \}
\\
&&+K_{n+1}\left ( x,x\right ) -2\left \{  K_{n+1}\left ( x,x\right ) +
\left ( y-x\right ) K_{n+1}^{(0,1)}\left ( x,x\right ) +\frac{1}{2}
\left ( y-x\right ) ^{2}K_{n+1}^{(0,2)}\left ( x,s\right ) \right \}
\\
&=&\left ( y-x\right ) ^{2}\left \{  K_{n+1}^{\left ( 1,1\right ) }
\left ( r,r\right ) +K_{n+1}^{\left ( 2,0\right ) }\left ( r,r\right )
-K_{n+1}^{(0,2)}\left ( x,s\right ) \right \}
\\
&=&O\left ( n^{3}\left ( y-x\right ) ^{2}\right ) .
\end{eqnarray*}%
Third,
\begin{equation*}
K_{n+1}^{\left ( 0,1\right ) }\left ( y,x\right ) -K_{n+1}^{\left ( 0,1
\right ) }\left ( x,x\right ) =O\left ( n^{3}\left ( y-x\right )
\right ) .
\end{equation*}%
Fourth, for some $t$ between $y,x$,
\begin{eqnarray*}
&&K_{n+1}^{\left ( 1,0\right ) }\left ( y,x\right ) -
\frac{K_{n+1}\left (
y,x\right ) -K_{n+1}\left ( x,x\right ) }{y-x}
\\
&=&K_{n+1}^{\left ( 1,0\right ) }\left ( y,x\right ) -K_{n+1}^{\left (
1,0\right ) }\left ( t,x\right ) =O\left ( n^{3}\left ( y-x\right )
\right ) .
\end{eqnarray*}%
Fifth, for some $r,\zeta ,s$ between $y,x$, with $r,s$ as above,
\begin{eqnarray*}
&&K_{n+1}^{\left ( 1,0\right ) }\left ( y,y\right ) -K_{n+1}^{\left ( 1,0
\right ) }\left ( y,x\right ) -
\frac{K_{n+1}\left ( y,y\right ) +K_{n+1}\left ( x,x\right )
-2K_{n+1}\left ( x,y\right ) }{y-x}
\\
&=&\left ( y-x\right ) K_{n+1}^{\left ( 1,1\right ) }\left ( y,\zeta
\right ) -\left ( y-x\right ) \left \{  K_{n+1}^{\left ( 1,1\right ) }
\left ( r,r\right ) +K_{n+1}^{\left ( 2,0\right ) }\left ( r,r\right )
-K_{n+1}^{(0,2)}\left ( x,s\right ) \right \}
\\
&=&\left ( y-x\right ) \left \{  K_{n+1}^{\left ( 1,1\right ) }\left ( y,
\zeta \right ) -K_{n+1}^{\left ( 1,1\right ) }\left ( r,r\right ) -K_{n+1}^{
\left ( 2,0\right ) }\left ( r,r\right ) +K_{n+1}^{(0,2)}\left ( x,s
\right ) \right \}
\\
&=&O\left ( n^{4}\left ( y-x\right ) ^{2}\right ) ,
\end{eqnarray*}%
by the estimates in {Lemma~\ref{lem3.2}}. Sixth, for some $\xi $ between
$x,y$,
\begin{eqnarray*}
&&K_{n+1}^{\left ( 1,1\right ) }\left ( y,x\right ) -
\frac{K_{n+1}^{\left (
0,1\right ) }\left ( y,x\right ) -K_{n+1}^{\left ( 0,1\right ) }\left ( x,x\right )
}{y-x}
\\
&=&K_{n+1}^{\left ( 1,1\right ) }\left ( y,x\right ) -K_{n+1}^{\left (
1,1\right ) }\left ( \xi ,x\right ) =O\left ( n^{4}\left ( y-x\right )
\right ) .
\end{eqnarray*}%
Then substituting all these into {(\ref{eq5.10})},
\begin{eqnarray*}
&&\Omega _{12}\Delta
\\
&=&\det \left [
\begin{array}{c@{\quad }c@{\quad }c}
O\left ( n\right ) & O\left ( n^{2}\left ( y-x\right ) \right ) & O
\left ( n^{2}\right )
\\
O\left ( n^{2}\left ( y-x\right ) \right ) & O\left ( n^{3}\left ( y-x
\right ) ^{2}\right ) & O\left ( n^{3}\left ( y-x\right ) \right )
\\
O\left ( n^{3}\left ( y-x\right ) \right ) & O\left ( n^{4}\left ( y-x
\right ) ^{2}\right ) & O\left ( n^{4}\left ( y-x\right ) \right )%
\end{array}%
\right ]
\\
&=&\left ( y-x\right ) ^{3}\det \left [
\begin{array}{c@{\quad }c@{\quad }c}
O\left ( n\right ) & O\left ( n^{2}\right ) & O\left ( n^{2}\right )
\\
O\left ( n^{2}\right ) & O\left ( n^{3}\right ) & O\left ( n^{3}
\right )
\\
O\left ( n^{3}\right ) & O\left ( n^{4}\right ) & O\left ( n^{4}
\right )%
\end{array}%
\right ] =O\left ( n^{8}\left ( y-x\right ) ^{3}\right ) .
\end{eqnarray*}%
Here we extracted factors of $y-x$ from the second and third rows, and
then the second column. It follows that as a polynomial in $u$,
$\left ( \Omega _{12}\Delta \right ) \left ( x,x+
\frac{u}{n\omega \left ( x\right ) }\right ) $ has a zero of multiplicity
at least $3$ at $0$. Then $\frac{\Omega _{12}\Delta }{u^{3}}$ is a polynomial
in $u$, and
$\frac{\Omega _{12}}{%
\sqrt{\Delta }}=\frac{\Omega _{12}\Delta }{u^{3}}\left (
\frac{u^{2}}{\Delta }%
\right ) ^{3/2}$, which is analytic in some neighborhood of $0$ that is
independent of $n,x,u$. The uniform convergence in {(\ref{eq5.5})} gives uniformly
for $u$ in compact subsets of $\mathbb{R}$,
\begin{eqnarray*}
&&\frac{\Omega _{12}}{\sqrt{\Delta }}\left (
\frac{1}{n\omega \left ( x\right )
}\right ) ^{2}
\\
&=&\left [
\frac{\Omega _{12}\Delta }{K_{n+1}\left ( x,x\right ) ^{3}}\left (
\frac{1}{n\omega \left ( x\right ) }\right ) ^{2}\right ]
\frac{K_{n+1}\left (
x,x\right ) ^{3}}{\Delta ^{3/2}}
\\
&=&\frac{H\left ( u\right ) }{\left ( 1-S^{2}\left ( u\right ) \right ) ^{3/2}}%
+o\left ( 1\right ) .
\end{eqnarray*}%
Also then, $H\left ( u\right ) $ necessarily has a zero of multiplicity
$\geq 3 $ at $0$. Finally, uniformly for $u$ in compact subsets of
$\mathbb{R}$,
\begin{eqnarray*}
&&\frac{\left \vert \Omega _{12}\right \vert }{\sqrt{\Delta }}
\arcsin \left (
\frac{\left \vert \Omega _{12}\right \vert }{\sqrt{\Omega _{11}\Omega _{22}}}%
\right ) \left ( \frac{1}{n\omega \left ( x\right ) }\right ) ^{2}
\\
&\leq &\frac{\left \vert \Omega _{12}\right \vert }{\sqrt{\Delta }}
\frac{\pi }{%
2}\left ( \frac{1}{n\omega \left ( x\right ) }\right ) ^{2}\leq C.\qedhere
\end{eqnarray*}%
\end{proof}

Now we can deduce the desired bound near the diagonal:%

\begin{proof}[Proof of {Lemma~\ref{lem2.4}}(b)]
Recall from {(\ref{eq2.6})} that
\begin{eqnarray*}
&&\left \vert \rho _{2}\left ( x,y\right ) \right \vert \left (
\frac{1}{n\omega \left ( x\right ) }\right ) ^{2}
\\
&\leq &\frac{1}{\pi ^{2}}\left ( \sqrt{
\frac{\Omega _{11}\Omega _{22}-\Omega _{12}^{2}}{\Delta }}+
\frac{\left \vert \Omega _{12}\right \vert }{\sqrt{\Delta }}\arcsin
\left (
\frac{\left \vert \Omega _{12}\right \vert }{\sqrt{\Omega _{11}\Omega _{22}}}
\right ) \right ) \left ( \frac{1}{n\omega \left ( x\right ) }%
\right ) ^{2}\leq C,
\end{eqnarray*}%
by {Lemma~\ref{lem5.1}}(a), (b) and {Lemmas~\ref{lem5.2}--\ref{lem5.3}}.
Next,\ from {(\ref{eq4.1})}, followed
by {(\ref{eq5.7})}, (with $u=0$ there)
%
%e5.13 #&#
\begin{equation}\label{eq5.11}
\frac{\rho _{1}\left ( x\right ) }{n\omega \left ( x\right ) }=
\frac{1}{\pi }%
\sqrt{
\frac{\Psi \left ( x\right ) }{K_{n+1}\left ( x,x\right ) ^{2}\left (
n\omega \left ( x\right ) \right ) ^{2}}}=\frac{1}{\sqrt{3}}+o\left ( 1
\right ) ,
\end{equation}%
and a similar asymptotic holds for $\rho _{1}\left ( y\right ) $. It follows
that
\begin{equation*}
\left \vert \rho _{2}\left ( x,y\right ) -\rho _{1}\left ( x\right )
\rho _{1}\left ( y\right ) \right \vert \left (
\frac{1}{n\omega \left ( x\right ) }%
\right ) ^{2}\leq C,
\end{equation*}%
which gives the result, since $\omega $ is positive and continuous in
$\left [ a,b\right ] $.
\end{proof}

\begin{proof}[Proof of {Lemma~\ref{lem2.5}}]
This follows directly from {(\ref{eq5.11})}.
\end{proof}

%s6 #&#
\section{Appendix - proof of {Lemma~\ref{lem2.2}}}
\label{sec6}

In this section, we prove {Lemma~\ref{lem2.2}}. The functions
$\rho _{2}\left ( x,y\right ) $ and $\rho _{1}\left ( x\right ) $ arising
in {(\ref{eq2.5})}, are called the second and the first intensities, or the two-point
and one-point correlation functions of zeros, see, e.g.,
\cite[pp. 7--8]{Houghetal2009}. By their defining properties, we have
\begin{equation*}
\mathbb{E}\left [ N_{n}\left ( \left [ a,b\right ] \right ) \right ] =
\int _{a}^{b}\rho _{1}\left ( x\right ) dx
\end{equation*}%
and
\begin{equation*}
\mathbb{E}\left [ N_{n}\left ( \left [ a,b\right ] \right ) \left ( N_{n}
\left ( \left [ a,b\right ] \right ) -1\right ) \right ] =\int _{a}^{b}
\int _{a}^{b}\rho _{2}\left ( x,y\right ) \text{ }dx\text{ }dy.
\end{equation*}%
Thus the variance of real zeros of random orthogonal polynomials in an
interval $\left [ a,b\right ] \subset \mathbb{R}$ can be written as in
{(\ref{eq2.5})} by completing the following steps:
\begin{eqnarray*}
&&\text{Var}\left [ N_{n}\left ( \left [ a,b\right ] \right ) \right ]
\\
&=&\mathbb{E}\left [ N_{n}\left ( \left [ a,b\right ] \right ) ^{2}-
\mathbb{E}%
\left [ N_{n}\left ( \left [ a,b\right ] \right ) \right ] ^{2}
\right ]
\\
&=&\mathbb{E}\left [ N_{n}\left ( \left [ a,b\right ] \right ) \left (
N_{n}\left ( \left [ a,b\right ] \right ) -1\right ) \right ] -
\mathbb{E}\left [ N_{n}\left ( \left [ a,b\right ] \right ) \right ] ^{2}+
\mathbb{E}\left [ N_{n}\left ( \left [ a,b\right ] \right ) \right ]
\\
&=&\int _{a}^{b}\int _{a}^{b}\rho _{2}\left ( x,y\right ) \text{ }dx
\text{ }%
dy-\int _{a}^{b}\int _{a}^{b}\rho _{1}\left ( x\right ) \rho _{1}
\left ( y\right ) \text{ }dx\text{ }dy+\int _{a}^{b}\rho _{1}\left ( x
\right ) dx
\\
&=&\int _{a}^{b}\int _{a}^{b}\left \{  \rho _{2}\left ( x,y\right ) -
\rho _{1}\left ( x\right ) \rho _{1}\left ( y\right ) \right \}
\text{ }dx\text{ }%
dy+\int _{a}^{b}\rho _{1}\left ( x\right ) dx.
\end{eqnarray*}

We follow the argument of \cite{GranvilleWigman2011} in several parts of
this proof. For $x,y\in {\mathbb{R}}$, define the random vector
\begin{equation*}
V=V(x,y):=\left ( G_{n}(x),G_{n}(y),G_{n}^{\prime }(x),G_{n}^{\prime }(y)
\right ) ^{T},
\end{equation*}%
and observe that the components of this vector are Gaussian random variables
satisfying
\begin{equation*}
{\mathbb{E}}[G_{n}(x)]={\mathbb{E}}[G_{n}^{\prime }(x)]=0,\ \mathrm{Var}%
[G_{n}(x)]=K_{n+1}(x,x)\ \text{and}\ \mathrm{Var}[G_{n}^{\prime }(x)]=K_{n+1}^{(1,1)}(x,x).
\end{equation*}%
The covariance matrix $\Sigma $ of $V$ is defined by
%
%e6.1 #&#
\begin{align}
\label{eq6.1}
\Sigma & =\Sigma (x,y)
\nonumber
\\
& :=%
\begin{bmatrix}
\mathrm{Var}[G_{n}(x)] & \mathrm{Cov}[G_{n}(x),G_{n}(y)] &
\mathrm{Cov}%
[G_{n}(x),G_{n}^{\prime }(x)] & \mathrm{Cov}[G_{n}(x),G_{n}^{\prime }(y)]
\\
\mathrm{Cov}[G_{n}(y),G_{n}(x)] & \mathrm{Var}[G_{n}(y)] &
\mathrm{Cov}%
[G_{n}(y),G_{n}^{\prime }(x)] & \mathrm{Cov}[G_{n}(y),G_{n}^{\prime }(y)]
\\
\mathrm{Cov}[G_{n}^{\prime }(x),G_{n}(x)] & \mathrm{Cov}[G_{n}^{
\prime }(x),G_{n}(y)] & \mathrm{Var}[G_{n}^{\prime }(x)] &
\mathrm{Cov}%
[G_{n}^{\prime }(x),G_{n}^{\prime }(y)]
\\
\mathrm{Cov}[G_{n}^{\prime }(y),G_{n}(x)] & \mathrm{Cov}[G_{n}^{
\prime }(y),G_{n}(y)] & \mathrm{Cov}[G_{n}^{\prime }(y),G_{n}^{
\prime }(x)] & \mathrm{Var}[G_{n}^{\prime }(y)]%
\end{bmatrix}
\nonumber
\\
& =%
\begin{bmatrix}
K_{n+1}(x,x) & K_{n+1}(x,y) & K_{n+1}^{(0,1)}(x,x) & K_{n+1}^{(0,1)}(x,y)
\\
K_{n+1}(x,y) & K_{n+1}(y,y) & K_{n+1}^{(0,1)}(y,x) & K_{n+1}^{(0,1)}(y,y)
\\
K_{n+1}^{(0,1)}(x,x) & K_{n+1}^{(0,1)}(y,x) & K_{n+1}^{(1,1)}(x,x) & K_{n+1}^{(1,1)}(x,y)
\\
K_{n+1}^{(0,1)}(x,y) & K_{n+1}^{(0,1)}(y,y) & K_{n+1}^{(1,1)}(x,y) & K_{n+1}^{(1,1)}(y,y)%
\end{bmatrix}%
,
\end{align}
exactly as in {(\ref{eq3.1})}. When $x=y$, the first row of $\Sigma $ is the same
as the second row, and hence $\det \Sigma =0$. Our first goal is to show
that $%
V $ has the multivariate normal distribution with mean zero and the covariance
matrix $\Sigma $ when $x\neq y$ and $n\geq 3$. This follows in a standard
way, e.g., from \cite[Corollary 16.2]{JacodProtter2003}, by proving that
$\Sigma $ is positive definite, which amounts to showing that
$\vec{v}^{T}\Sigma \vec{v}>0$ for all nonzero $\vec{v}\in {\mathbb{R}}^{4}$. Recall
that any covariance matrix is positive semi-definite
\cite[Theorem 12.4]{JacodProtter2003}, i.e., $\vec{v}^{T}\Sigma \vec{v}\geq 0$ for all
$\vec{v}%
\in {\mathbb{R}}^{4}$. This means we only need to demonstrate that
$\vec{v}%
^{T}\Sigma \vec{v}=0$ implies $\vec{v}=\vec{0}$. For a vector
$\vec{v}=%
\begin{bmatrix}
v_{1} & v_{2} & v_{3} & v_{4}%
\end{bmatrix}%
^{T}$, observe that
\begin{equation*}
\vec{v}^{T}\Sigma \vec{v}=\mathrm{Var}[\vec{v}^{T}V]=%
\sum _{k=0}^{n}(v_{1}p_{k}(x)+v_{2}p_{k}(y)+v_{3}p_{k}^{\prime }(x)+v_{4}p_{k}^{
\prime }\left ( y\right ) )^{2}.
\end{equation*}%
It is clear now that $\vec{v}^{T}\Sigma \vec{v}=0$ if and only if
%
%e6.2 #&#
\begin{equation}\label{eq6.2}
v_{1}p_{k}(x)+v_{2}p_{k}(y)+v_{3}p_{k}^{\prime }(x)+v_{4}p_{k}^{
\prime }(y)=0,\quad k=0,\ldots ,n.
\end{equation}%
But this system of equations has only trivial solution
$\vec{v}=\vec{0}$. Indeed, if we write
\begin{equation*}
Q_{n}(t)=\sum _{j=0}^{n}b_{j}p_{j}(t),
\end{equation*}%
where $\{b_{j}\}_{j=0}^{n}\subset {\mathbb{R}}$ is arbitrary, then {(\ref{eq6.2})}
implies that
%
%e6.3 #&#
\begin{equation}\label{eq6.3}
v_{1}Q_{n}(x)+v_{2}Q_{n}(y)+v_{3}Q_{n}^{\prime }(x)+v_{4}Q_{n}^{
\prime }(y)=0.
\end{equation}%
Since $\{p_{j}(x)\}_{j=0}^{n}$ is a basis for the vector space of all polynomials
of degree at most $n$ with real coefficients, the set of all polynomials
$Q_{n}(t)$ coincides with this space. In particular, since
$%
n\geq 3$ and $x\neq y$, we use the following choices for $Q_{n}$ in {(\ref{eq6.3})}
to conclude that
\begin{align*}
Q_{n}(t)=(t-x)(t-y)^{2}& \Rightarrow v_{3}=0;
\\
Q_{n}(t)=(t-x)^{2}(t-y)& \Rightarrow v_{4}=0;
\\
Q_{n}(t)=t-y& \Rightarrow v_{1}=0;
\\
Q_{n}(t)=t-x& \Rightarrow v_{2}=0.
\end{align*}%
We now write $\Sigma $ in the following block form
%
%e6.4 #&#
\begin{equation}\label{eq6.4}
\Sigma =%
\begin{bmatrix}
K_{n+1}(x,x) & K_{n+1}(x,y) & K_{n+1}^{(0,1)}(x,x) & K_{n+1}^{(0,1)}(x,y)
\\
K_{n+1}(x,y) & K_{n+1}(y,y) & K_{n+1}^{(0,1)}(y,x) & K_{n+1}^{(0,1)}(y,y)
\\
K_{n+1}^{(0,1)}(x,x) & K_{n+1}^{(0,1)}(y,x) & K_{n+1}^{(1,1)}(x,x) & K_{n+1}^{(1,1)}(x,y)
\\
K_{n+1}^{(0,1)}(x,y) & K_{n+1}^{(0,1)}(y,y) & K_{n+1}^{(1,1)}(x,y) & K_{n+1}^{(1,1)}(y,y)%
\end{bmatrix}%
=:%
\begin{bmatrix}
A & B
\\
B^{T} & C%
\end{bmatrix}%
,
\end{equation}%
where $A$, $B$ and $C$ are the corresponding $2\times 2$ matrices. Note
that $\det A=\Delta =0$ if and only if $x=y$ by the equality case in the
Cauchy-Schwarz inequality. Thus we define $\Omega =C-B^{T}A^{-1}B$ for
$%
x\neq y$, and write
\begin{equation*}
\Sigma =%
\begin{bmatrix}
A & \mathbf{0}
\\
B^{T} & \mathbf{I}%
\end{bmatrix}%
\begin{bmatrix}
\mathbf{I} & A^{-1}B
\\
\mathbf{0} & \Omega
\end{bmatrix}%
.
\end{equation*}%
The latter implies that
\begin{equation*}
\det \Sigma =\det A\,\det \Omega =\Delta \,\det \Omega .
\end{equation*}%
Since $\Sigma $ is invertible for $x\neq y$, so is $\Omega $ and thus
$\det \Omega >0$ if $x\neq y$. It also follows from {(\ref{eq6.4})} by direct algebraic
manipulations that the elements of the matrix
\begin{equation*}
\Omega =C-B^{T}A^{-1}B=%
\begin{bmatrix}
\Omega _{11} & \Omega _{12}
\\
\Omega _{12} & \Omega _{22}%
\end{bmatrix}%
\end{equation*}%
are as defined in {(\ref{eq2.8})}--{(\ref{eq2.10})}.

Since the random vector $V=V(x,y)$ has the multivariate normal distribution
$%
\mathcal{N}(\mathbf{0},\Sigma )$ with a non-singular covariance matrix
$%
\Sigma $, we compute the density of its distribution by
\cite[p. 130]{JacodProtter2003} in the form
\begin{align*}
p_{x,y}(0,0,t_{1},t_{2})& =
\frac{\exp {\left ( -\frac{1}{2}%
(0,0,t_{1},t_{2})\,\Sigma ^{-1}(0,0,t_{1},t_{2})^{T}\right ) }}{(2\pi )^{2}(\det \Sigma )^{1/2}}
\\
& =
\frac{\exp {\left ( -\frac{1}{2}(t_{1},t_{2})\,\Omega ^{-1}(t_{1},t_{2})^{T}\right ) }}{(2\pi )^{2}(\det \Sigma )^{1/2}}.
\end{align*}%
Using matrix algebra, we further obtain that
\begin{equation*}
\Sigma ^{-1}=%
\begin{bmatrix}
\lbrack A-BC^{-1}B^{T}]^{-1} & -A^{-1}B[C-B^{T}A^{-1}B]^{-1}
\\
-C^{-1}B^{T}[A-BC^{-1}B^{T}]^{-1} & [C-B^{T}A^{-1}B]^{-1}%
\end{bmatrix}%
.
\end{equation*}%
Theorem 3.2 of \cite[p. 71]{AzaisWschebor2009} states that if
$(a,b)\subset {\mathbb{R}}$, then
\begin{equation*}
{\mathbb{E}}[N_{n}(\left [ a,b\right ] )\left ( N_{n}(\left [ a,b
\right ] )-1\right ) ]=\iint _{D}\int _{{\mathbb{R}}}\int _{{
\mathbb{R}}%
}|t_{1}t_{2}|p_{x,y}(0,0,t_{1},t_{2})\,dt_{1}dt_{2}dxdy,
\end{equation*}%
where $D=\{(x,y)\in {\mathbb{R}}^{2}|\ a\leq x,y\leq b\}$. Hence
\begin{align*}
& {\mathbb{E}}[N_{n}(\left [ a,b\right ] )(N_{n}(\left [ a,b\right ] )-1)]
\\
& =\iint _{D}\int _{{\mathbb{R}}}\int _{{\mathbb{R}}}|t_{1}t_{2}|
\frac{\exp {%
\left ( -\frac{1}{2}(t_{1},t_{2})\Omega ^{-1}(t_{1},t_{2})^{T}\right ) }}{%
(2\pi )^{2}(\det \Sigma )^{1/2}}\,dt_{1}dt_{2}dxdy,
\\
& =\iint _{D}\int _{{\mathbb{R}}}\int _{{\mathbb{R}}}|t_{1}t_{2}|
\frac{\exp {%
\left ( -\frac{1}{2}(t_{1},t_{2})\Omega ^{-1}(t_{1},t_{2})^{T}\right ) }}{%
(2\pi )^{2}(\Delta \,\det \Omega )^{1/2}}\,dt_{1}dt_{2}dxdy,
\\
& =\frac{1}{4\pi ^{2}}\iint _{D}
\frac{I(x,y)}{\sqrt{\Delta \,\det \Omega }}%
\,dxdy,
\end{align*}%
where the inner integral is
\begin{equation*}
I(x,y)=\int _{{\mathbb{R}}}\int _{{\mathbb{R}}}|t_{1}t_{2}|\exp {
\left ( -\frac{%
1}{2}(t_{1},t_{2})\Omega ^{-1}(t_{1},t_{2})^{T}\right ) }\,dt_{1}dt_{2}.
\end{equation*}%
Note if $x\neq y$, we have
$\det \Omega =\Omega _{11}\Omega _{22}-\Omega _{12}^{2}>0$ and
\begin{equation*}
\Omega ^{-1}=\frac{1}{\det \Omega }%
\begin{bmatrix}
\Omega _{22} & -\Omega _{12}
\\
-\Omega _{12} & \Omega _{11}%
\end{bmatrix}%
.
\end{equation*}%
It follows that
\begin{equation*}
(t_{1},t_{2})\,\Omega ^{-1}\,(t_{1},t_{2})^{T}=
\frac{\Omega _{22}}{\det \Omega }\,t_{1}^{2}-2\,
\frac{\Omega _{12}}{\det \Omega }\,t_{1}t_{2}+
\frac{%
\Omega _{11}}{\det \Omega }\,t_{2}^{2}.
\end{equation*}%
Applying the result of \cite[(3.9)]{BleherDi1997}, we evaluate the inner
integral as
\begin{equation*}
I(x,y)=
\frac{4(\det \Omega )^{2}}{\Omega _{11}\Omega _{22}(1-\delta ^{2})}%
\left ( 1+\frac{\delta }{\sqrt{1-\delta ^{2}}}\arcsin \delta \right ) ,
\end{equation*}%
with
\begin{equation*}
\delta =-\frac{\Omega _{12}}{\sqrt{\Omega _{11}\Omega _{22}}}.
\end{equation*}%
Finally, putting everything together, we obtain
\begin{align*}
& {\mathbb{E}}[N_{n}(\left [ a,b\right ] )(N_{n}(\left [ a,b\right ] )-1)]
\\
& =\frac{1}{4\pi ^{2}}\iint _{D}
\frac{4(\det \Omega )^{2}}{\Omega _{11}\Omega _{22}(1-\delta ^{2})}
\left ( 1+\frac{\delta }{\sqrt{1-\delta ^{2}}}\arcsin \delta \right )
\frac{dx\,dy}{\sqrt{\Delta \,\det \Omega }}
\\
& =\frac{1}{\pi ^{2}}\iint _{D}\sqrt{\Omega _{11}\Omega _{22}-\Omega _{12}^{2}%
}\left ( 1-
\frac{\Omega _{12}}{\sqrt{\Omega _{11}\Omega _{22}-\Omega _{12}^{2}%
}}\arcsin \left ( -
\frac{\Omega _{12}}{\sqrt{\Omega _{11}\Omega _{22}}}%
\right ) \right ) \frac{dx_{1}\,dx_{2}}{\sqrt{\Delta }}
\\
& =\frac{1}{\pi ^{2}}\iint _{D}\left ( \sqrt{\Omega _{11}\Omega _{22}-
\Omega _{12}^{2}}+\Omega _{12}\arcsin
\frac{\Omega _{12}}{\sqrt{\Omega _{11}\Omega _{22}}}\right )
\frac{dx\,dy}{\sqrt{\Delta }}.
\end{align*}%
This and {Lemma~\ref{lem2.1}} give the result.
\qed

\end{document}